\documentclass[a4paper,leqno]{amsart}
\usepackage{cite,mathrsfs}
\usepackage{amssymb,times}
\usepackage{pstricks,pst-node,multido}
\usepackage[vcentermath]{youngtab}

\def\tab(#1){\mbox{\tiny$\young(#1)$}\,}

\title[Carter-Payne homomorphisms]%
      {Carter-Payne homomorphisms and Jantzen filtrations}

\author[S.~Lyle]{Sin\'ead Lyle}
\address{School of Mathematics \\ University of East Anglia \\ Norwich NR4 7TJ \\ UK}
\author[A.~Mathas]{Andrew Mathas}
\address{School of Mathematics and Statistics \\ University of Sydney \\ NSW 2006 \\ Australia}
\keywords{Hecke algebras, Carter-Payne homomorphisms, Jantzen filtrations.}

\numberwithin{equation}{section}
\numberwithin{figure}{section}
\newtheorem{lemma}[equation]{Lemma}
\newtheorem{theorem}[equation]{Theorem}
\newtheorem{proposition}[equation]{Proposition}
\newtheorem{corollary}[equation]{Corollary}

\newtheorem*{theorem*}{Theorem}
\newtheorem*{proposition*}{Proposition}
\newtheorem*{lemma*}{Lemma}
\newtheorem*{definition*}{Definition}
\newtheorem*{caution*}{Caution}

\theoremstyle{definition}

\newtheorem*{notation}{Notation}

\newtheorem{conjecture}[equation]{Conjecture}

\theoremstyle{remark}
\newtheorem*{remark}{Remark}
\newenvironment{Example}[1][\relax]%
  {\refstepcounter{equation}\trivlist
   \item[\hskip\labelsep\theequation.~\textbf{Example#1}\space]
   \ignorespaces
  }{\unskip\nobreak\hfil%
    \penalty50\hskip2em\hbox{}\nobreak\hfil$\Diamond$%
    \parfillskip=0pt\finalhyphendemerits=0\penalty-100\endtrivlist}

\newenvironment{claim}[1][\relax]
  {\trivlist\item[\hskip\labelsep\bf Claim\ifx#1\relax\relax\else\space#1\fi.]\it}
  {\endtrivlist}

\def\({\big(}
\def\){\big)}
\let\<=\langle
\let\>=\rangle
\let\gdom=\vartriangleright
\let\gedom=\trianglerighteq
\let\bar=\overline

\def\a{\mathbf{a}}
\def\A{\mathcal{A}}
\newcommand{\Diag}{\mathbb D}
\def\O{\mathscr O}
\def\val{\mathpzc{val}\!}
\newcommand{\h}[1][n]{\mathscr{H}_{#1}}
\newcommand{\hlam}[1][\lambda]{\mathscr{H}^{\gdom#1}}
\newcommand\HO[1][n]{\h[#1]^{\O}}
\newcommand\HK[1][n]{\h[#1]^{K}}
\newcommand{\sym}[1][n]{\mathfrak{S}_{#1}} 
\def\Sym{\mathfrak S}

\newcommand{\C}{\mathbb{C}}
\newcommand{\Q}{\mathbb{Q}}

\newcommand{\N}{\mathbb{N}}
\newcommand{\M}{\mathpzc{M}}
\newcommand{\Z}{\mathbb{Z}}
\newcommand{\calZ}{\mathcal{Z}}

\newcommand{\la}{\lambda}

\let\Sect=\S

\def\S{\mathsf S}
\def\T{\mathsf T}
\def\U{\mathsf U}
\def\V{\mathsf V}
\def\D{\mathcal D}
\newcommand{\ten}{10}
\newcommand{\eleven}{11}
\newcommand{\twelve}{12}
\renewcommand{\t}{\mathfrak{t}}
\newcommand{\tnu}{\t^\nu}
\newcommand{\s}{\mathfrak{s}}

\def\calL{\mathcal L}
\newcommand{\LL}{L_{\lambda\mu}}
\newcommand{\LLsm}{L_{\sigma\mu}}
\newcommand{\LLls}{L_{\lambda\sigma}}
\newcommand{\LLO}{L^\O_{\lambda\mu}}
\newcommand{\LLOp}{L^{\prime\O}_{\lambda\mu}}
\newcommand{\LLOpp}{L^{\prime\prime\O}_{\lambda\mu}}

\newcommand{\q}{q}

\newcommand{\inim}[1]{\overset{#1}\longleftarrow}
\DeclareMathOperator{\Shape}{Shape}
\def\bump_#1(#2){b_{#1}^{#2}}
\DeclareMathOperator{\row}{row}
\DeclareMathOperator{\Std}{Std}
\DeclareMathOperator{\rowstd}{RStd}
\def\SStd{\mathcal T_0}

\DeclareMathOperator{\Hom}{Hom}
\DeclareMathOperator{\End}{End}
\DeclareMathOperator{\im}{Im}
\DeclareMathOperator{\res}{c}

\DeclareMathAlphabet{\mathpzc}{OT1}{pzc}{m}{it}
\newcommand{\EHom}[1]{\mathpzc{Hom}_{\hspace*{-0.05cm}\mathscr{H}_{#1}}}

\def\map#1#2{\,{:}\,#1\!\longrightarrow\!#2}

{\catcode`\|=\active
  \gdef\set#1{\mathinner{\lbrace\,{\mathcode`\|"8000%
                                   \let|\midvert #1}\,\rbrace}}
}
\def\midvert{\egroup\mid\bgroup}

\begin{document}
\begin{abstract}
We prove a $q$-analogue of the Carter-Payne theorem in the case where the
differences between the parts of the partitions are sufficiently large. We
identify a layer of the Jantzen filtration which contains the image of
these Carter-Payne homomorphisms and we show how these
homomorphisms compose.
\end{abstract}
\maketitle

\section{Introduction}
The Iwahori-Hecke algebras of the symmetric groups are interesting algebras
with a rich combinatorial representation theory. These algebras arise
naturally in the representation theory of the general linear groups and they are
important because they simultaneously extend and generalize the
representation theory of the symmetric and general linear groups.

The representation theory of the Hecke algebra $\h$ closely parallels that
of the symmetric groups.  For each partition~$\lambda$ of~$n$ there is a
Specht module $S^\lambda$.  In the semisimple case the Specht modules give
a complete set of pairwise non-isomorphic irreducible $\h$-modules.  When
$\h$ is not semisimple it is an important problem to determine the
structure of the Specht modules. The purpose of this paper is to construct
explicit non-trivial homomorphisms between Specht modules in the
non-semisimple case.  Using this construction, we are then able to 
connect the image of the homomorphism and the Jantzen filtration of the 
corresponding Specht module.  

The most striking result about homomorphisms between Specht modules of the
symmetric groups is the \textit{Carter-Payne Theorem}~\cite{Payne}, which
was proved by building on the famous paper of Carter and
Lusztig~\cite{CarterLusztig}. A second proof of the Carter-Payne Theorem
has recently been given by Fayers and Martin~\cite{FayersMartin:homs}.  

In this paper we are concerned with the Carter-Payne homomorphisms of
the Iwahori-Hecke algebra of the symmetric group.  To state our main
results, let $F$ be a field of characteristic $p\ge0$ and fix a
non-zero element $\zeta\in F$. Let $e>1$ be minimal such that
$1+\zeta+\dots+\zeta^{e-1}=0$; set $e=0$ if no such integer exists.
Let $\h$ be the Hecke algebra of the symmetric group $\mathfrak{S}_n$,
over $F$, with parameter $\zeta$.

If $p>0$ and $k>0$ then define $\ell_p(k)$ to be the smallest positive
integer such that $p^{\ell_p(k)}>k$.  Now suppose that $\gamma>0$ and
$\lambda$ and $\mu$ are partitions of $n$ 
such that
\[\mu_i = \begin{cases}
\la_i+\gamma,& i=a, \\
\la_i-\gamma,&  i=z, \\
\la_i, & \text{otherwise},
\end{cases}\]
for some positive integers $a<z$. Let $h=\lambda_a-\lambda_z+z-a+\gamma$.
Then $\lambda$ and $\mu$ form an \textbf{$\boldsymbol(e,p)$-Carter-Payne
pair}, with parameters $(a,z,\gamma)$, if $e>1$ and either
\begin{enumerate}
  \item $p=0$, $\gamma<e$ and $h\equiv0\pmod e$, or,
  \item $p>0$ and $h\equiv0\pmod{ep^{\ell_p(\gamma^*)}}$, where
    $\gamma^*=\lfloor\frac\gamma e\rfloor$.
\end{enumerate}

The Carter-Payne Theorem for an Iwahori-Hecke algebra of the symmetric
group is the following result.

\begin{theorem}[Carter and Payne~\cite{Payne} and Dixon~\cite{Dixon}]
  \label{CarterPayne}
Suppose that $F$ is a field of characteristic $p\ge0$ and
that $\lambda$ and $\mu$ form an $(e,p)$-Carter-Payne pair. Then 
$\Hom_{\h}(S^\lambda,S^\mu)\ne0$.
\end{theorem}

For the symmetric groups (that is, when $q=1$) this theorem is a classical
result of Carter and Payne~\cite{Payne}. The full $q$-analogue of this
result for the Iwahori-Hecke algebra $\h$ was recently established in the
unpublished thesis of Dixon~\cite{Dixon}. Dixon's proof follows the
original arguments of Carter and Lusztig~\cite{CarterLusztig} and Carter
and Payne~\cite{Payne}. He works with the quantum hyperalgebra $U$ of the
general linear group and he proves that if $\lambda$ and $\mu$ form an
$(e,p)$-Carter-Payne then $\Hom_U(\Delta^\la,\Delta^\mu)$ is one
dimensional, where $\Delta^\nu$ is the Weyl module of~$U$ indexed by the
partition $\nu$. As
$\dim\Hom_{\h}(S^\lambda,S^\mu)\ge\dim\Hom_U(\Delta^\la,\Delta^\mu)$, with
equality if $q\ne-1$, this implies Theorem~\ref{CarterPayne} for
arbitrary~$q$.

The Carter-Payne homomorphisms are very useful and important maps.
Unfortunately little is known about them in general except that they exist.
In this paper we concentrate on separated Carter-Payne pairs, where an
$(e,p)$-Carter-Payne pair $(\lambda,\mu)$ with parameters
$(a,z,\gamma)$ is \textbf{separated} if $\la_r - \la_{r+1} \geq \gamma$ for
$a < r \leq z$.  We begin by giving two new and very explicit descriptions
of Carter-Payne homomorphisms $\theta_{\la\mu}\map{S^\la}S^\mu$ when $\la$
and $\mu$ form a separated Carter-Payne pair. We then use the new
descriptions to prove the following two results, which were known previously
only for the symmetric group algebra when $\gamma=1$~\cite{EM}. 

\begin{theorem} \label{Composition}
Suppose that $\lambda$, $\mu$ and $\sigma$ are
partitions of $n$ such that $\lambda$ and $\sigma$ form a separated
$(e,p)$-Carter-Payne pair with parameters $(a,y,\gamma)$ and that
$\sigma$ and $\mu$ form a separated $(e,p)$-Carter-Payne pair with parameters
$(y,z,\gamma)$, where $a<y<z$ and $\gamma>0$.  Then $\la$ and $\mu$ form a separated $(e,p)$-Carter-Payne pair with parameters $(a,z,\gamma)$ and
$\theta_{\la\sigma}\theta_{\sigma\mu}=\theta_{\la\mu}$.
\end{theorem}

To state our next result let
$S^\mu=J^0(S^\mu)\supseteq J^1(S^\mu)\supseteq J^2(S^\mu)\supseteq\dots$ be
the Jantzen filtration of $\S^\mu$ (see Section~\ref{JantzenFiltrations}),
and for $0\ne h\in\Z$ define
$$\val_{e,p}(h)=\begin{cases}
  p^{\val_p(h)},&\text{if }e\mid h,\\
  0,&\text{otherwise,}
\end{cases}$$
where $\val_p$ is the usual $p$-adic valuation map (and we set
$\val_0(h)=0$ when $p=0$). Our second main result is the following.

\begin{theorem}\label{Jantzen}
  Suppose that $p\ge0$ and that $\lambda$ and $\mu$ form a separated
  $(e,p)$-Carter-Payne pair with parameters $(a,z,\gamma)$. Then
  $$\im\theta_{\la\mu}\subseteq J^{\delta}(S^\mu),$$
  where $\delta=\val_{e,p}(\la_a-\la_z+z-a+\gamma)-\val_{e,p}(\gamma)$.
\end{theorem}

The key observation in our construction of the Carter-Payne homomorphisms,
 which is due to Ellers and Murray~\cite{EllersMurray}, 
is that the Specht modules $S^\la$ and
$S^\mu$ both appear in the restriction of a Specht module~$S^\nu$
of~$\h[n+\gamma]$. Starting from this observation we are able show that the
Carter-Payne homomorphism $\theta_{\la\mu}\map{S^\la}S^\mu$ is induced by
an $\h$-module endomorphism of $S^\nu$ which is given by right
multiplication by a polynomial in the Jucys-Murphy elements
$L_{n+1},\dots,L_{n+\gamma}$ of~$\h[n+\gamma]$. Using this description of
the Carter-Payne maps we are able to prove the two theorems above as well
as describe these maps as explicit linear combinations of semistandard
homomorphisms.  Thus we give a new proof of Theorem~\ref{CarterPayne}, 
when $\la$ and $\mu$ are a separated pair, which
takes place entirely within the Hecke algebra.  

We now describe the contents of this paper in more detail.  Section~2 sets
up the basic notation and machinery that is used throughout the paper.
In Theorem~\ref{CPDistinctZero} and Theorem~\ref{CPDistinctPrime} we
show that if $(\la,\mu)$ is a separated $(e,p)$-Carter-Payne pair
then the corresponding
Carter-Payne homomorphism is given by right multiplication by a
polynomial in the Jucys-Murphy elements of $\h[n+\gamma]$. We prove
these results by writing the Carter-Payne homomorphism
$\theta_{\la\mu}$ as an explicit linear combination of semistandard
homomorphisms. These results are proved modulo a result which
describes how the Jucys-Murphy elements act on the Specht modules
(Proposition~\ref{ImageDistinct}) and a technical result which allows
us to divide our maps by certain polynomial coefficients when $p>0$
(Lemma~\ref{divides1}). Using these results we prove our two main
theorems about composing Carter-Payne homomorphisms and the connection
between these maps and the Jantzen filtration. Section~3 is the
computational heart of the paper which proves the detailed technical
results which describe the action of the Jucys-Murphy elements on the
Specht modules which are need to prove our main theorems. The
results in this section are likely to be of independent interest.

\section{Carter-Payne homomorphisms and Jucys-Murphy elements} 
In this section we define the Hecke algebra and the Specht modules and
reduce the proofs of our main results to some technical statements
which are proved in the next section.

\subsection{The Hecke algebra}\label{Hecke}
For each integer $n>0$ let $\sym[n]$ be the symmetric group of
degree~$n$. The symmetric group $\sym[n]$ is generated by the simple
transpositions $s_1,s_2,\dots,s_{n-1}$, where $s_i=(i,i+1)$ for 
$1\le i<n$. If $w\in\sym[n]$ then $s_{i_1}\dots s_{i_k}$ is a
\textbf{reduced expression} for~$w$ if $w=s_{i_1}\dots s_{i_k}$ and
$k$ is minimal with this property. In this case, the \textbf{length}
of $w$ is~$\ell(w)=k$.

Suppose that $\q$ is an indeterminate over $\Z$ and let $\calZ=\Z[\q,\q^{-1}]$
be the ring of Laurent polynomials in~$\q$. The 
\textbf{generic Iwahori-Hecke algebra} of $\sym[n]$ is the unital associative
$\calZ$-algebra $\h^\calZ$ with generators
$T_1,\dots,T_{n-1}$ which are subject to the relations
$$ (T_i-q)(T_i+1)=0, \quad T_jT_{j+1}T_j=T_{j+1}T_jT_{j+1}
       \quad\text{and}\quad T_iT_j=T_jT_i,$$
where $1\le i<n$, $1\le j< n-1$ and $|i-j|\ge 2$. The Hecke algebra
$\h$ is free as an $\calZ$-module with basis $\set{T_w|w\in\Sym_n}$,
where $T_w=T_{i_1}\dots T_{i_k}$ and $s_{i_1}\dots s_{i_k}$ is a
reduced expression for $w$; see, for example,
\cite[Chapt.~1]{M:ULect}. 

Now suppose that $R$ is an arbitrary ring and that $q_R$ is an
invertible element of~$R$. Define $\h^R(q_R)=\h^\calZ\otimes_\calZ R$,
where we consider $R$ as a $\calZ$-algebra by letting $q$ act as
multiplication by $q_R$. We say that $\h$ is obtained from $\h^\calZ$
by \textbf{specialization} at $\q=q_R$.  By the remarks above,
$\h^R(q_R)$ is a unital associative $R$-algebra which is free as an
$R$-module with basis $\set{T_w\otimes 1|w\in\sym}$. Typically, we
abuse notation and write $T_w$ instead of~$T_w\otimes 1$, for
$w\in\sym$. 

In this paper we are most interested in the algebra $\h=\h^F(\zeta)$,
where $F$ is a field of characteristic $p\ge0$ and $0\ne\zeta\in F$.
Define
$$e=\min\set{f\ge2|1+\zeta+\dots+\zeta^{f-1}=0},$$
and set $e=0$ if $1+\zeta+\dots+\zeta^{f-1}\ne0$ for all $f\ge2$.
Then $\h$ is (split) semisimple if and only if $e>n$ or $e=0$; see, for
example, \cite[Cor.~3.24]{M:ULect}). Henceforth, we assume that 
$2\leq e\le n$. In particular, $\h$ is not semisimple. 

Observe that if $\zeta=1$ then $e=p$ and $\h\cong F\sym[n]$. If
$\zeta\ne1$ then $\zeta$ is a primitive $e^{\text{th}}$ root of unity
in~$F$.  

\subsection{Tableaux combinatorics}\label{combinatorics}
A \textbf{composition} of~$n$ is a sequence
$\lambda=(\lambda_1,\lambda_2,\dots)$ of non-negative integers which sum
to~$n$ and $\lambda$ is a \textbf{partition} if
$\lambda_1\ge\lambda_2\ge\dots$. The diagram of a partition $\lambda$ is the set
$\Diag(\lambda)=\set{(r,c)|1\le c\le\lambda_r, \text{for }r\ge1}.$
A (row standard) \textbf{$\lambda$-tableau} is a map 
$\S\map{\Diag(\lambda)}\N$ such that $\S(r,c)\le\S(r,c')$, whenever 
$c\le c'$. We identify a $\lambda$-tableau with a labeling of the
diagram of $\lambda$ by~$\N$, and in this way we can talk of the rows and
columns of~$\S$.  A $\lambda$-tableau $\S$ is:
\begin{enumerate}
  \item \textbf{semistandard} if the entries in $\S$ are strictly
    increasing down columns.  
  \item \textbf{standard} if $\S\map{\Diag(\la)}\{1,2,\dots,n\}$ is a
    bijection and the entries in $\S$ are strictly increasing down columns;
\end{enumerate}
A $\lambda$-tableau has \textbf{type} $\mu=(\mu_1,\mu_2,\dots)$ if it
has $\mu_i$ entries equal to~$i$, for $i\ge1$. If $\S$ is a
$\lambda$-tableau let $\Shape(\S)=\lambda$ and if $k\geq 0$ let 
$\S_{\downarrow k}$ be the subtableau of $\S$ containing the numbers
$1,2,\dots,k$.

The following notation will help us keep track of certain entries in our tableaux. 

\begin{notation}
Suppose that $\S$ is a tableau and that $X$ and $R$ are sets of
positive integers. Let~$\S^X_R$ be the number of entries in row $r$
of~$\S$ which are equal to some $x$, for some $r\in R$ and some 
$x\in X$.  We further abbreviate this notation by setting 
$\S^{\le x}_{>r}=\S_{(r,\infty)}^{[1,x]}$, $\S^x_r=\S^{\{x\}}_{\{r\}}$ 
and so on.
\end{notation}

Let $\mathcal{T}(\lambda,\mu)$ be the set of $\lambda$-tableau of type 
$\mu$ and $\SStd(\lambda,\mu)$ the set of semistandard 
$\lambda$-tableaux of type $\mu$.  Let $\Std(\la)=\SStd(\la,(1^n))$ 
be the set of standard tableaux and $\rowstd(\la)=\mathcal{T}(\la,(1^n))$ 
the set of tableaux of type $(1^n)$. The \textbf{initial $\la$-tableau}
is the standard $\la$-tableau $\t^\la$ obtained by entering the numbers
$1,2,\ldots,n$ in increasing order, from left to right, along the rows 
of~${\Diag(\lambda)}$.

If $\s$ is a tableau and $\s(r,c)=k$ then define $\row_\s(k)=r$.  For
any subset $I \subseteq \set{1,2,\ldots,n}$, the entries in $I$
are in \textbf{row order} in $\s$ if $\row_\s(i)\le \row_\s(j)$
whenever $i<j\in I$.  For example,~$\t^\la$ is the unique
$\la$-tableau which has $1,2,\dots,n$ in row order.

There is an action of $\Sym_n$ on $\rowstd(\lambda)$, from the right,
given by defining $\s w$ to be the $\lambda$-tableau obtained from
$\s$ by acting on the entries of $\s$ by $w$ and then reordering the
entries in each row, for $\s\in\rowstd(\la)$ and $w\in\Sym_n$. If
$\s\in\rowstd(\la)$ define $d(\s)$ to be the unique element
of~$\Sym_n$ of minimal length such that $\s=\t^\la d(\s)$; such an element
exists, for example, by \cite[Prop.~3.3]{M:ULect}.  The permutation $d(\s)$ 
is the unique element of $\sym$ such that  $\s=\t^\la d(\s)$ and
$(i)d(\s)<(j)d(\s)$ whenever $i<j$ lie in the same row of $\s$. Let 
$\Sym_\la=\Sym_{\la_1}\times\Sym_{\la_2}\times \ldots$ be the
\textbf{Young subgroup} of $\sym$ associated to $\la$.

\subsection{Specht modules}
For each pair of tableaux $\s,\t\in\Std(\lambda)$, for $\lambda$ a
partition of~$n$, let $m_{\s\t}=T_{d(\s)^{-1}}m_\lambda T_{d(\t)}$,
where
$$m_\la = \sum_{w\in\Sym_\la}T_w.$$
Murphy showed that $\set{m_{\s\t}|\s,\t\in\Std(\la), \text{ for $\la$ a partition
of }n}$ is a basis of $\h$~\cite{Murphy,M:ULect}. The basis
$\{m_{\s\t}\}$ is a cellular basis of $\h$ with respect to the dominance
ordering where if $\la$ and $\mu$ are partitions then $\mu \gedom \la$ if
$$\sum_{i=1}^j\mu_i\ge\sum_{i=1}^j\la_i,$$
for all $j\ge1$. Write $\mu\gdom\la$ if $\mu \gedom \la$ and $\mu\ne \la$. Let $\hlam$ be the two-sided ideal of $\h$ with basis
$\set{m_{\s\t}| \s,\t\in\Std(\mu) \text{ for some }\mu\gdom\la}$.

Fix a partition $\lambda$ of $n$. The \textbf{Specht module}
$S^\lambda_F$ is the $\h$-submodule of $\h/\hlam$ generated by
$m_\la+\hlam$. For every tableau $\s\in\rowstd(\la)$ define 
$m_\s= m_\la T_{d(\s)}+\hlam$. Then $m_\s\in S^\lambda$ and
$\set{m_\t|\t\in\Std(\lambda)}$ is a basis of $S^\la_F$ by, for
example, \cite[Prop.~3.22]{M:ULect}. This construction of the Specht
module works over an arbitrary ring. In particular, we have a Specht
module $S^\la_\calZ$ for the generic Hecke algebra $\h^\calZ$ and
$S^\la_F\cong S^\la_\calZ\otimes_\calZ F$ as $\h$-modules. Usually, we write
$S^\la=S^\la_F$ unless we want to emphasize the ring that $S^\la$ is
defined over. 

We emphasize, for the readers convenience, that throughout this paper we
follow~\cite{M:ULect} and work with the Specht modules that arise as the
cell modules for the Murphy basis.  These modules are \textit{dual} to the
classically defined Specht modules considered in~\cite{DJ:reps,James}. Our
results can be translated into the corresponding results for the classical
Specht modules by conjugating the partitions involved and taking duals;
see, for example, \cite[Lemma~3.4]{LM:rowhoms}.

Define the \textbf{Jucys-Murphy elements} $L_1,\dots,L_{n}$ of
$\h$ by setting $L_1=q^{-1}T_1$ and
$L_{k+1}=q^{-1}T_k(1+L_kT_k)$. Then $L_1,\dots,L_{n}$ generate
a commutative subalgebra of $\h$; see, for example,
\cite[Prop.~3.26]{M:ULect}. The Jucys-Murphy elements
$L_k$ are important for us because 
they act on the Specht modules via triangular matrices.

If $R$ is any ring, $a\in R$ and $k\in\Z$ then the 
\textbf{Gaussian integer} $[k]_a$ is 
$$[k]_a=\lim_{t\to a}\tfrac{t^k-1}{t-1},$$
where $t$ is an indeterminate over~$R$.  If $k\ge0$ set $[0]_a^!=1$
and let $[k]_a^!=[k-1]_a^![k]_a$. We are most interested in these
scalars when $R=\calZ$ and $r=\q$, so we set $[k]=[k]_\q$ and
$[k]^!=[k]_\q^!$, for $k\in\Z$.

\subsection{Constructing Carter-Payne homomorphisms}\label{ProofSection}
Suppose that $\la$ is a partition of~$n$ and let $M^\la=m_\la\h$ be the
corresponding permutation module for~$\h$. Then
$M^\la$ has basis $\set{m_\la T_{d(\t)}|\t\in\rowstd(\la)}$ and
there is a surjective homomorphism $\pi_\la\map{M^\la}{S^\la}$ given by
$\pi_\la(m_\la T_{d(\t)})=m_\t$, for $\t\in\rowstd(\la)$. 

Now if $\mu$ is a partition of $n$ and $\t \in \Std(\mu)$, define
$\la(\t)$ to be the $\mu$-tableau obtained by replacing each entry 
in $\t$ by its row index in $\t^\la$. If $\T$ is a $\mu$-tableau of type
$\la$ define
\[m_\T =\sum_{\substack{\t \in \rowstd(\mu)\\\la(\t)=\T}} m_\t.\] 
By definition $m_\T\in S^\mu$.  

If $\T\in\SStd(\mu,\lambda)$ let $\varphi_\T\in\Hom_{\h}(M^\lambda,S^\mu)$ be the homomorphism determined by $\varphi_\T(m_{\la})=m_\T$ and let $\EHom{n}(M^\lambda,S^\mu)$ be the subspace of $\Hom_{\h}(M^\lambda,S^\mu)$ spanned by $\set{\varphi_\T | T \in \SStd(\mu,\la)}$.
Let $\EHom{n}(S^\lambda,S^\mu)$ be the space of homomorphisms $\varphi \in \Hom_{\h}(S^\lambda,S^\mu)$ such that $\pi_\la \varphi \in \EHom{n}(M^\lambda,S^\mu)$.

To reprove Theorem~\ref{CarterPayne} for the separated $(e,p)$-Carter-Payne pair
$(\la,\mu)$ we use the following result. This is purely a matter of
notational convenience as the proof that we give can be made to work
without making use of this proposition (cf. the proof of
Theorem~\ref{Composition} in Section~\ref{CompositionSection} below).

\begin{proposition} \label{RowRemoval}
Suppose that $\la$ and $\mu$ are partitions of $m$ such that
$\la_i=\mu_i$, whenever $1\le i<a$ or $i>z$, for some integers $a<z$.
Define $\hat{\la}=(\la_a,\la_{a+1},\ldots,\la_z)$ and
$\hat{\mu}=(\mu_a,\mu_{a+1},\ldots,\mu_z)$ and let
$n=\hat\la_a+\dots+\hat\la_z=\hat\mu_a+\dots+\hat\mu_z$.
Then $$\EHom{m}(S^\la,S^\mu)\cong_F \EHom{n}(\S^{\hat{\la}},\S^{\hat{\mu}}).$$
\end{proposition}

\begin{proof} 
This follows from (the proof of)~\cite[Theorem 3.2 and Lemma 3.4]{LM:rowhoms};
cf.~\cite[Prop.~10.4]{Donkin:tilting}.
\end{proof}

Therefore, when constructing Carter-Payne homomorphisms it is
enough to show that $\EHom{n}(\S^{\la},\S^{\mu})\ne0$ for
partitions $\lambda$ and $\mu$ of~$n$ which form a
separated $(e,p)$-Carter-Payne pair with parameters $a=1$,
$z=\max\set{i>0|\lambda_i\ne0}$ and $\gamma>0$.   
For the rest of Section \ref{ProofSection} we
fix such a pair. 
We define $\nu$ to be the partition of~$n+\gamma$ given by
$$\nu_i = \begin{cases} \la_i+\gamma, & \text{if }i=1, \\
                        \la_i, & \text{otherwise.}
\end{cases}$$

There is a natural embedding $\h\hookrightarrow\h[n+\gamma]$. Thus we
can consider any $\h[n+\gamma]$-module as an $\h$-module by restriction.
We need the following well-known result -- it is an easy corollary of
\cite[Prop.~6.1]{M:ULect}.    

\begin{lemma}\label{restriction} As an $\h$-module the Specht module
  $S^\nu$ has a filtration
  $$S^\nu=M_0\supset M_1\supset\dots\supset M_k\supset 0,$$
  such that $M_i/M_{i+1}\cong S^{\tau_i}$, for some partition $\tau_i$ 
  of~$n$, for $0\le i\le k$. Moreover $S^\lambda\cong M_0/M_1$, 
  $S^\mu\cong M_k$, $M_1$ has basis 
  $\set{m_\t|\t\in\Std(\nu)\text{ and }\Shape(\t_{\downarrow n})\ne\lambda}$, 
  and $M_k$ has basis 
  $\set{m_\t|\t\in\Std(\nu)\text{ and }\t_{\downarrow n}\in\Std(\mu)}$.
\end{lemma}

Hence, following Ellers and Murray~\cite[\Sect~3]{EllersMurray}, we
have the following elementary but very useful observation.

\begin{corollary}\label{EllersMurray}
  Let $\nu$ be the partition of $n+\gamma$ defined above and suppose that
  there exists a non-zero homomorphism $\theta\in \End_{\h}(S^\nu)$
  such that $M_1\subseteq\ker(\theta)$ and 
  $\im(\theta)\subseteq M_k$.  
  Then $\Hom_{\h}(S^\la,S^\mu)\ne0$.  
\end{corollary} 

Set $c_i = \nu_i -i$, for $1\le i\le z$.
(Thus, $c_i$ is the content of the $i^{\text{th}}$ removable node of~$\nu$; see the remarks before Lemma~\ref{permutation}.) Now define
\[\LL=\prod_{i=1}^{z-1}\prod_{j=1}^{\gamma}\Big(L_{n+j}-[c_i]\Big).\]

\begin{lemma}\label{MurphyCommutes}
Suppose that $1\le k<n+\gamma$ and $k\ne n$. 
Then $T_k \LL = \LL T_k$.
\end{lemma}

\begin{proof}
  By Lemma~\ref{Murphyformulae}, below, if $k\ne n$ then $T_k$ commutes with
  $(L_{n+1}-c)\dots(L_{n+\gamma}-c)$, for any $c\in F$.
\end{proof}

Hence, right multiplication by $\LL$ induces an $\h$-endomorphism of
$S^\nu$.  The following definitions allow us to describe this map and
(if necessary) to modify it so as to produce an endomorphism $\theta_{\la\mu}$
which satisfies the conditions of Lemma \ref{EllersMurray}.

Let $\eta$ be a partition of $n$. Write $\eta\subseteq\nu$ if
$\Diag(\eta)\subseteq\Diag(\nu)$; equivalently, $\eta_i\le\nu_i$, for
$i\ge1$. If $\eta \subseteq \nu$, define $\eta+1^\gamma=(\eta_1,\ldots,\eta_k,0^{z-k},1^\gamma)$, where $k=\max\set{i|\eta_i >0}$.  Define~$\tnu_\eta$ to be the standard tableau which agrees with $\t^\eta$ where $\Diag(\eta)$ and $\Diag(\nu)$ coincide, with the numbers $n+1,\ldots,n+\gamma$ entered in row order in the remaining nodes of $\Diag(\nu)$. 
A $\nu$-tableau $\t$ is \textbf{almost initial} if $\t=\tnu_\eta$, for
some partition~$\eta$ of $n$. 

Now suppose that $\eta$ is a partition of $n$ such that $\eta\subseteq\nu$.
Define
$$\SStd^\nu(\mu,\eta)=\set{\S\in\SStd(\nu,\eta+1^\gamma)|
\Shape(\S_{\downarrow z})=\mu}.$$
That is, $\SStd^\nu(\mu,\eta)$ is the set of semistandard $\nu$-tableaux of
type $\eta+1^\gamma$ obtained by adding nodes labeled
$z+1,\dots,z+\gamma$ to the bottom of a semistandard $\mu$-tableaux of type $\eta$.
Similarly, if $\eta\subseteq\nu$ let
$\Std_\eta(\nu)=\set{\t\in\Std(\nu)|\Shape(\t_{\downarrow n})=\eta}$.

The following elegant result will allow us to construct Carter-Payne
homomorphisms. It will be proved with less elegance in the following
sections. The number $\S_r^{(r,z]}$, which is the number of entries in
row~$r$ of $\S$ contained in $(r,z]$, is defined in
Section~\ref{combinatorics}.

\begin{proposition} \label{ImageDistinct}
  Suppose that $\eta\subset\nu$ is a partition of $n$.  Then there
  exists an integer $C$ such that in $S^\nu_\calZ$
  $$m_{\tnu_\eta}\LL =q^{C}\sum_{\S\in\SStd^\nu(\mu,\eta)}
       \prod_{r=1}^{z-1}\Big( [\S_r^{(r,z]}]^!
       \prod_{j=0}^{\gamma-\S_r^{(r,z]}-1}[c_z-c_r-j]\Big) m_\S.$$
\end{proposition}

\begin{proof}
    This is the special case of Proposition~\ref{Image} below, obtained
    by setting~$k=\gamma$ and~$y=1$. 
\end{proof}

\begin{Example}\label{example1}
Suppose that $\la=(4,4,2)$ and $\mu=(6,4)$. Then $\lambda$ and $\mu$
form a $(6,0)$-Carter-Payne pair with parameters
$(a,z,\gamma)=(1,3,2)$. Applying the definitions,
\[\LL=(L_{12}-[5])(L_{12}-[2])(L_{11}-[5])(L_{11}-[2]).\]
Identifying the tableau $\S$ with $m_{\S}$, direct computation (or 
Proposition~\ref{ImageDistinct}) shows that
{\small
\begin{align*}
{\tab(1234\eleven\twelve,5678,9\ten)} \;\LL = 
{\tab(111145,2222,33)} \;\LL & = q^{2}[2][2] \;{\tab(111122,2233,45)}-q^{-1}[2][3] \;{\tab(111123,2223,45)}\\
        &\hspace*{26.5mm}+q^{-5}[2][3][4] \;{\tab(111133,2222,45)}, \\
{\tab(12345\eleven,678\twelve,9\ten)} \;\LL = 
{\tab(111114,2225,33)} \;\LL &= -q^{-2}[2][6]\;{\tab(111112,2233,45)}+q^{-5}[3][6] \;{\tab(111113,2223,45)}, \\
{\tab(12345\eleven,6789,\ten\twelve)} \;\LL = 
{\tab(111114,2222,35)} \;\LL &= q^{-2}[3][6]\;{\tab(111112,2223,45)} -q^{-6} [3][4][6]\;{\tab(1111113,2222,45)}, \\
{\tab(123456,78\eleven\twelve,9\ten)} \;\LL = 
{\tab(111111,2245,33)}\;\LL  &= q^{-5}[2][6][7]\;{\tab(111111,2233,45)}, \\
{\tab(123456,789\eleven,\ten\twelve)} \;\LL = 
{\tab(111111,2224,35)} \;\LL &= -q^{-6}[3][6][7]\;{\tab(111111,2223,45)}, \\
{\tab(123456,789\ten,\eleven\twelve)} \;\LL = 
{\tab(111111,2222,45)} \;\LL &= q^{-6}[3][4][6][7]\;{\tab(111111,2222,45)}.
\end{align*}
}%
Using these calculations we invite the reader to check that right
multiplication by $\LL$ induces an $\h[10]^{\C}$-module homomorphism 
$S^\lambda\to S^\mu$ when $\zeta=\exp(2\pi i/6)\in\C$ (so that $e=6$). 
\end{Example}

Using Proposition \ref{ImageDistinct}, we can now give a second proof of
Theorem~\ref{CarterPayne} from the introduction for our pair $(\la,\mu)$. 
We treat the cases $p=0$ and $p>0$ separately because the 
proof when $p>0$ contains an additional subtlety.

\begin{theorem} \label{CPDistinctZero}
  Suppose that $p=0$ and that $\lambda$ and $\mu$ form a separated
  $(e,0)$-Carter-Payne pair with parameters $(a,z,\gamma)$. Then
  $$\EHom{n}(S^\lambda,S^\mu)\ne0.$$
\end{theorem}

\begin{proof}By Proposition~\ref{RowRemoval} it is enough to consider the
  case when $a=1$ and $\lambda_r=0$ when $r>z$. Since $\lambda$ and $\mu$
  form a Carter-Payne pair we have, by assumption, that $\gamma<e$ and
  $$\lambda_1-\lambda_z+z-1+\gamma= c_1-c_z\equiv0\pmod e.$$
  In particular, $[c_z-c_1]_{\zeta}=0$ in $F$.
  
  Suppose that $\t\in\Std(\nu)$ and let $\eta=\Shape(\t_{\downarrow n})$. Then 
  in $S^\nu$ we have $m_\t=m_{\tnu_\eta}T_w$, for some
  $w\in\Sym_n\times\Sym_\gamma$. Therefore, by specializing $q=\zeta$
  in Proposition~\ref{ImageDistinct} and using Lemma~\ref{MurphyCommutes}, we
  have 
  $$m_\t \LL = \zeta^{C}\sum_{\S\in\SStd^\nu(\mu,\eta)} \prod_{r=1}^{z-1}
         \Big(\,[\S_r^{(r,z]}]_\zeta^!\prod_{j=0}^{\gamma-\S_r^{(r,z]}-1}[c_z-c_r-j]_\zeta
         \Big)\, m_\S T_w,$$ 
  for some $C\in\Z$. Recall that as an $\h$-module $S^\nu$
  has a Specht filtration 
  $S^\nu=M_0\supset M_1\supset\dots\supset M_k\supset0$ with
  $S^\lambda\cong M_0/M_1$ and $S^\mu\cong M_k$ by
  Lemma~\ref{restriction}. Moreover, $M_k$ is spanned by the $m_\s$,
  for $\s\in\Std(\nu)$ with $\Shape(\s_{\downarrow n})=\mu$.
  Therefore, the last displayed equation shows that $m_\t\LL\in M_k$
  for $\t\in\Std(\nu)$. 

  Next suppose that $\t\in \Std_\eta(\nu)$ and $m_\t\in M_1$. Then
  $\eta\ne\lambda$ by Lemma~\ref{restriction}. Consequently, if
  $\S\in\SStd^\nu(\mu,\eta)$ then $\S_1^{(1,z]}<\gamma$ and
  $[c_z-c_1]_\zeta$ divides the coefficient of $m_\S$ in
  $m_\t\LL$. That is, $m_\t\LL=0$ since $[c_z-c_1]_\zeta=0$
  in~$F$.
 
  By the last two paragraphs, and Lemma~\ref{EllersMurray}, right
  multiplication by $\LL$ induces an $\h$-module homomorphism
  from $S^\lambda$ to $S^\mu$.  Suppose that $\t=\tnu_\lambda$. Then
  there exists a semistandard tableau $\S\in\SStd^\nu(\mu,\lambda)$
  with $\S_r^{(r,z]}=\gamma$, for $1\le r<z$. This is the unique semistandard
  tableau $\S\in\SStd^\nu(\mu,\lambda)$ such that row $r$ contains
  $\gamma$ entries equal to $r+1$, for $1\le r <z$. The coefficient
  of $m_\S$ in $m_{\tnu_\lambda}\LL$ is $\zeta^{C}\([\gamma]_\zeta\)^{z-1}\ne0$,
  so that $m_{\tnu_\lambda}\LL\ne0$ as required.

  We have now shown that right multiplication on $S^\nu$ by $\LL$ induces a 
  non-zero map $\theta_{\lambda\mu}\map{S^\lambda}S^\mu$. It remains to show
  that $\theta_{\lambda\mu}\in\EHom{n}(S^\lambda,S^\mu)$. However, from what we
  have proved it follows that 
  $$\pi_\la \theta_{\lambda\mu} = \zeta^{C}\sum_{\S\in\SStd(\mu,\lambda)} 
      \prod_{r=1}^{z-1} \Big( [\S_r^{(r,z]}]_\zeta^!
      \prod_{j=0}^{\gamma-\S_r^{(r,z]}-1}[c_z-c_r-j]_\zeta\Big)
      \, \varphi_\S.$$ 
  Therefore, $\theta_{\lambda\mu}\in\EHom{n}(S^\lambda,S^\mu)$ as
  claimed.
\end{proof}

We now consider the case when $p>0$. The argument is essentially the
same as in the case when $p=0$. There is a technical difficulty,
however, because in general multiplication by $\LL$ induces the zero
homomorphism from $S^\la$ to $S^\mu$.

\begin{theorem} \label{CPDistinctPrime}
  Suppose that $p>0$ and that $\lambda$ and $\mu$ form a separated $(e,p)$-Carter-Payne pair
  with parameters $(a,z,\gamma)$. Then
  $$\EHom{n}(S^\lambda,S^\mu)\ne0$$
\end{theorem}

\begin{proof}As in Theorem~\ref{CPDistinctZero}, we may assume that $a=1$ and $\la_r=0$ 
    for $r>z$. We first consider the Specht module $S^\nu_\calZ$ for the
    generic Hecke algebra $\h[n+\gamma]^\calZ$ defined over
    $\calZ=\Z[\q,\q^{-1}]$. Suppose that $\t\in\Std(\nu)$ and set
    $\eta=\Shape(\t_{\downarrow n})$ so that $m_\t = m_{\tnu_\eta}T_w$
    for some $w \in \sym[n]\times \Sym_{\gamma}$.  By Lemma~\ref{MurphyCommutes} and
    Proposition~\ref{ImageDistinct} in $S^\nu_\calZ$ we have
    $$m_\t \LL = \sum_{\S\in\SStd^\nu(\mu,\eta)} \q^{C} \prod_{r=1}^{z-1}
         \Big( [\S_r^{(r,z]}]_\q^!\prod_{j=0}^{\gamma-\S_r^{(r,z]}-1}[c_z-c_r-j]_\q\Big) 
             m_\S T_w.$$ 
    If $\gamma<e$ then, as in the proof of Theorem~\ref{CPDistinctZero}, there exists a tableau $\S$ with coefficient $\zeta^C [\gamma]_\zeta^{z-1} \neq 0$ when we specialize at $\q=\zeta$. Therefore, in this case we
    set $\q=\zeta$ and argue exactly as in the proof of
    Theorem~\ref{CPDistinctZero} to show that multiplication by $\LL$
    induces a non-zero homomorphism in
    $\EHom{n}(S\lambda,S^\mu)\ne0$. If $\gamma\ge e$ then we have to
    work harder because the coefficients on the right hand side are
    almost always zero when we specialize to $\h[n+\gamma]$.  

  Suppose $1\le r <z$.  By Lemma~\ref{divides1} below, there exists an
  integer $\beta_r$ with $0\le \beta_r\le \gamma$ such that for all
  integers $\delta$ with $0\le\delta\le\gamma$ there exist polynomials
  $f_{r,\delta}(\q)$ and $g_{r,\delta}(\q)$ in~$\calZ$, which depend only on
  $c_z-c_r$, such that $g_{r,\delta}(\zeta)\ne0$ and
  $$ \frac{[\delta]^!\prod_{j=0}^{\gamma-\delta-1}[c_z-c_r-j]_\q}
      {[\beta_r]^!\prod_{j=0}^{\gamma-\beta_r-1}[c_z-c_r-j]_\q} 
           = \frac{f_{r,\delta}(\q)}{g_{r,\delta}(\q)}.$$
   Hence, there is a well-defined $\h^\calZ$-module homomorphism 
   $\theta^\calZ_{\la\mu} \in \End_{\h}(S_\calZ^\nu)$ given by
   $\theta^\calZ_{\lambda\mu}(h)=\frac1{\beta_{\lambda\mu}}h\LL$, for all $h\in S_\calZ^\nu$,
   where
   $$\beta_{\lambda\mu}=\beta_{\lambda\mu}(\q)=\prod_{r=1}^{z-1}\Bigg([\beta_r]_\q^!
          \prod_{j=0}^{\gamma-\beta_r-1}[c_z-c_r-j]_\q\Bigg)\cdot
          \prod_{\delta=0}^{\gamma}\prod_{r=1}^{z-1}\frac1{g_{r,\delta}(\q)}.$$
   Since $\lambda$ and $\mu$ form an $(e,p)$-Carter-Payne pair, we have
   $c_z-c_1\equiv0\pmod{ep^{\ell_p(\gamma^\ast)}}$, where
   $\gamma^*=\lfloor\frac\gamma e\rfloor$. Consequently, by
   Lemma~\ref{divides1}, $\beta_1 =\gamma$ and $f_{1,\delta}(\zeta)\ne 0$ if
   and only if $\delta=\beta_1$. Therefore, arguing as in the proof of
   Theorem~\ref{CPDistinctZero}, we see that specializing at $\q=\zeta$ gives a
   $\h$-module homomorphism $\theta_{\lambda\mu}\map{S^\lambda}S^\mu$ such that
   $$\pi_\la \theta_{\lambda\mu}= \sum_{\S\in\SStd^\nu(\mu,\la)} \Bigg(\zeta^{C} 
                    \prod_{r=1}^{z-1}f_{r,\S_r^{(r,z]}}(\zeta) 
		    \prod_{\delta=0}^\gamma g_{r,\delta}(\zeta)\Bigg)
		    \,\varphi_\S.$$ 

  Finally, to show that $\theta_{\lambda\mu}$ is non-zero we show that there
  exists a tableau $\S\in\SStd^\nu(\mu,\la)$ such that $\S_r^{(r,z]} =
  \beta_r$, for $1\le r<z$. This is enough because for such a tableau
  $\S$ the paragraphs above show that~$m_\S$ appears in
  $\theta_{\lambda\mu}(m_{\tnu_\la})$ with coefficient
  $\prod_{r=1}^{z-1}\prod_{\delta=0}^\gamma g_{r,\delta}(\zeta)$,
  and this is non-zero by construction. 

  In general, there are many tableaux $\S\in\SStd^\nu(\mu,\la)$ with
  $\S_r^{(r,z]} = \beta_r$, for $1\le r<z$. To construct a family of
  tableaux with this property set
  $\beta_z=\gamma$.  For $1 \leq r \leq z$ we construct a partition
  $\nu^{(r)}$ and a semistandard  $\nu^{(r)}$-tableaux~$\S^{(r)}$ of
  type $(\lambda_1,\dots,\lambda_r)$ with the properties that 
  $(\S^{(r)})_k^k=\nu_k-\beta_k$, for $1\le k\le r$, and
  $$\nu^{(r)}_1+\dots+\nu^{(r)}_r=\nu_1+\dots+\nu_r-\gamma.\leqno(\dag)$$
  To start, let $\S^{(1)}$ be the unique semistandard
  $(\lambda_1)$-tableau of type $(\lambda_1)$. By induction we may assume
  that we have constructed a semistandard $\nu^{(r)}$-tableau $\S^{(r)}$ as 
  above. Now define $\S^{(r+1)}$ to be any $\nu^{(r+1)}$-tableau of type
  $(\la_1,\dots,\la_{r+1})$ which is obtained by adding $\lambda_{r+1}$
  entries labeled $r+1$ to $\S^{(r)}$ in such a way that
  $\nu^{(r+1)}\subset\nu$ and $\nu^{(r+1)}_{r+1}=\nu_{r+1}-\beta_{r+1}$.
  Such tableaux exist because of $(\dag)$ since
  $\beta_{r+1}\le\gamma\le\nu_{r+1}$.  The tableau $\S^{(r+1)}$ is
  semistandard because $\lambda_i-\lambda_{i+1}\ge\gamma$, for $1\le i\le
  r$. It is easy to check that $\S^{(r+1)}$ satisfies all of the properties
  that we assumed of $\S^{(r)}$, so proceeding in this way we can construct a 
  semistandard $\nu^{(z)}$-tableaux of type
  $\lambda=(\lambda_1,\dots,\lambda_z)$. In fact,
  $\nu^{(z)}=\mu$ by $(\dag)$ because, by construction,
  $(\S^{(z)})_z^z=\nu_z-\beta_z=\mu_z$ since $\beta_z=\gamma$.
  Therefore, if we define $\S=\S^{(z+1)}$ to be the tableau obtained by 
  adding entries labeled $z+1,\dots,z+\gamma$ in row order to
  row $z$ of~$\S^{(z)}$ then $\S\in\SStd^\nu(\mu,\la)$ and 
  $\S_r^{(r,z]}=\beta_r$, for $1\le r<z$. Consequently, the coefficient of 
  $m_\S$ in $\theta_{\lambda\mu}(m_{\tnu_\la})$ is non-zero, so
  $\theta_{\la\mu}\ne0$ as claimed.
\end{proof}

If $p>0$ let $\beta_{\la\mu}(q)\in\Z[q]$ be the polynomial defined during
the proof of Theorem~\ref{CPDistinctZero} and if $p=0$ set
$\beta_{\la\mu}(q)=1$. Then the Carter-Payne homomorphisms
$\theta_{\la\mu}\map{S^\la}S^\mu$ that we constructed in the proofs of
Theorem~\ref{CPDistinctZero} and Theorem~\ref{CPDistinctPrime} are both of
the form
\begin{equation}\label{CP hom}
\theta_{\la\mu}(m_\t) =\frac1{\beta_{\la\mu}(\zeta)} m_\t\LL, 
\end{equation}
for $\t\in\Std_\la(\nu)$ (and this expression makes sense).

\begin{Example}
  As in Example~\ref{example1}, suppose that $\lambda=(4,4,2)$ and
  $\mu=(6,4)$. Then~$\lambda$ and~$\mu$ form an $(e,p)$-Carter-Payne pair
  with $e=2$ and $p=3$. Dividing all of the equations in
  Example~\ref{example1} by $[2]=1+q$ we obtain a map
  $\theta_{\lambda\mu}\map{S^{(4,4,2)}}S^{(6,4)}$. In fact, the calculations
  in Example~\ref{example1} show that $\pi_\la \theta_{\lambda\mu}=\varphi_\S$ where
\[\S=\tab(111123,2223).\]
However, applying Lemmas 5 and 7 from \cite[\Sect2]{FayersMartin:homs} it is
possible to show that if $e=p=2$ then 
$$\dim\text{Hom}_{\h[10]}(S^{(4,4,2)},S^{(6,4)})=1.$$ 
The existence of such a map is not predicted by the Carter-Payne
theorem. Moreover, looking at Example~\ref{example1} shows that this
map is not induced by right multiplication by any multiple of~$\LL$
because in order to make this map non-zero we need to divide by
$[2]_\zeta$ but then 
$${\tab(111114,2225,33)} \;\frac{1}{[2]_\zeta}\LL = {\tab(111113,2223,45)} \ne0,$$
when we set $\zeta=-1$. Consequently, right multiplication by
$\LL/[2]_\zeta$
does not induce a homomorphism from $S^\lambda$ to $S^\mu$ when
$e=p=2$ because, using the notation of Lemma~\ref{restriction}, the
submodule~$M_1$ of~$S^\nu$ is not killed by $\LL$.
\end{Example}

\subsection{Composing Homomorphisms} \label{CompositionSection}
This section shows that we can compose certain Carter-Payne
homomorphisms. This gives a positive answer to a
question of Henning Andersen.

Recall that Theorems~\ref{CPDistinctZero} and~\ref{CPDistinctPrime}
construct a non-zero homomorphism $\theta_{\la\sigma}\map{S^\la}S^\sigma$
whenever $\la$ and $\sigma$ form a separated Carter-Payne pair with parameters
$(a,y,\gamma)$. Let $\mu$ be another partition of $n$ and suppose
that $a<y<z$. Then it is easy to see that $\la$ and $\mu$ form a
separated Carter-Payne pair with parameters $(a,z,\gamma)$ if and only if
$\sigma$ and $\mu$ form a separated Carter-Payne pair with parameters
$(y,z,\gamma)$. Thus we have two homomorphisms $\theta_{\la\mu}$ and
$\theta_{\la\sigma}\theta_{\sigma\mu}$, which may be the zero map, from
$S^\la$ to $S^\mu$.

\begin{theorem} \label{composition}
Suppose that $\lambda$, $\mu$ and $\sigma$ are
partitions of $n$ such that $\lambda$ and $\sigma$ form a separated
$(e,p)$-Carter-Payne pair with parameters $(a,y,\gamma)$ and that
$\sigma$ and $\mu$ form a separated $(e,p)$-Carter-Payne pair with parameters
$(y,z,\gamma)$, where $a<y<z$ and $\gamma>0$.  Then 
$\theta_{\la\mu} = \theta_{\la\sigma}\theta_{\sigma\mu}$.
\end{theorem}

\begin{proof}
Using Proposition~\ref{RowRemoval}, we may assume that $a=1$
    and $z=\max\set{i>0|\la_i\ne0}$.  Let $\nu$ be the
    partition of $n+\gamma$ given by
    $$\nu_i=\begin{cases}\lambda_i+\gamma,&\text{if }i=1,\\
                               \lambda_i,&\text{otherwise}.
    \end{cases}$$
    Then $\lambda,\mu,\sigma\subset\nu$. 

    To prove the Theorem we consider the Specht module $S^\nu_\calZ$ for the 
    generic Iwahori-Hecke algebra~$\h[n+\gamma]^\calZ$.
    As in Lemma~\ref{restriction}, we fix a Specht filtration
    $$S_\calZ^\nu=M_0\supset M_1\supset\dots\supset M_l\supset
                        \dots\supset M_k\supset0$$
    of $S_\calZ^\nu$ such that, as $\h^\calZ$-modules, 
    $S_\calZ^\lambda\cong M_0/M_1$,
    $S_\calZ^\mu \cong M_k$ and $S_\calZ^\sigma\cong M_l/M_{l+1}$ for some 
    $1 \leq l <k$. We may assume that 
    $\set{m_\t|\sigma \not \gedom \Shape(\t_{\downarrow n})}$ is a basis of $M_{l+1}$.  
    For $1\le i\le z$, set  $c_i = \nu_i - i$. Mirroring the
    definition of $\LL$ (see before Lemma~\ref{MurphyCommutes}), set
    $$\LLls=\prod_{i=1}^{y-1}\prod_{j=1}^\gamma(L_{n+j}-[c_i]_q)
    \qquad\text{and}\qquad
    \LLsm=\prod_{i=y}^{z-1}\prod_{j=1}^\gamma(L_{n+j}-[c_i]_q).$$
    Then $\LL=\LLls\LLsm$. By (\ref{CP hom}) there exist 
    polynomials $\beta_{\la\mu}(q), \beta_{\sigma\mu}(q)\in\Z[q]$ such that 
    $$\theta_{\la\mu}(m_\t)=\frac1{\beta_{\la\mu}(\zeta)}m_\t\LL,$$
    for $\t\in\Std_\la(\nu)$. Via Proposition~\ref{RowRemoval}, we have
    analogous descriptions of the maps $\theta_{\la\sigma}$ 
    and~$\theta_{\sigma\mu}$, however, we do not (yet) have a
    description of these maps as $\h$-module endomorphisms of $S^\nu$.
    The next three claims allow us to describe these maps as
    endomorphisms of $S^\nu$ and to connect them with $\theta_{\la\mu}$.

    \begin{claim}[1]
      Suppose that $\eta$ is a partition of $n$ such that $\eta\subset\nu$
      and $\eta\gedom\sigma$. Then, in $S^\nu_\calZ$, 
      $$m_{\tnu_\eta} \LLsm=\q^{C_1} \sum_{\S\in\SStd^\nu(\mu,\eta)} 
        \prod_{r=y}^{z-1}\Big( [\S_r^{(r,z]}]_\q^!
        \prod_{j=0}^{\gamma-\S_r^{(r,z]}-1}[c_z-c_r-j]_\q\Big)\, m_\S.$$ 
    for some $C_1\in\Z$. Moreover, if $\S\in\SStd^\nu(\mu,\eta)$ then
    $\S^r_r=\mu_r$,  for $1\le r\le y$.
    \end{claim}

    \begin{proof}[Proof of Claim 1] When $y=1$ this is precisely
      Proposition~\ref{ImageDistinct}. We are assuming, however, that $y>1$. 
      In this case, the formula for $m_{\tnu_\eta}\LLls$ follows by setting 
      $k=\gamma$ in Proposition~\ref{Image} below (which includes 
      Proposition~\ref{ImageDistinct} as a special case). Secondly,
      observe that $\row_{\tnu_\eta}(n+j)\ge y$, for $1\le j\le\gamma$, 
      because $\eta\gedom\sigma$. Consequently, if~$\S\in\SStd^\nu(\mu,\eta)$ 
      then $\eta_r=\S^r_r=\mu_r$,  for $1\le r\le y$.
    \end{proof}

    \begin{claim}[2]
      Suppose that $\eta$ is a partition of $n$ such that $\eta\subset\nu$.
      Then, in $S^\nu_\calZ/M_{l+1}$, 
      $$m_{\tnu_\eta} \LLls \equiv \q^{C_2} \sum_{\S\in\SStd^\nu(\sigma,\eta)}
        \prod_{r=1}^{y-1}\Big( [\S_r^{(r,y]}]_\q^!
        \prod_{j=0}^{\gamma-\S_r^{(r,y]}-1}
                      [c_y-c_r-j]_\q\Big)\, m_\S \pmod{M_{l+1}}.$$ 
    for some $C_2\in\Z$. Moreover, if $\S\in\SStd^\nu(\sigma,\eta)$ then
    $\S^r_r=\sigma_r$,  for $y\le r\le z$.
    \end{claim}

    \begin{proof}[Proof of Claim 2] First observe that, by
      Lemma~\ref{onnj} below, $m_{\tnu_\eta}\LLls$ is a linear combination
      of terms $m_\s$ where $\s_{\downarrow n}\gedom\t^\eta$. If
      $\row_\s(n+j)>y$ for some $j$ with $1\le j\le\gamma$ then $m_\s\in
      M_{l+1}$, so we may assume that $\row_\s(n+j)\le y$ for $1\le
      j\le\gamma$. Consequently, if $m_\S+S_{l+1}$ appears with
      non-zero coefficient in $m_{\tnu_\eta}\LLls$ for 
      some~$\S\in\SStd^\nu(\mu,\eta)$ 
      then $\eta_r=\S^r_r=\sigma_r$,  for $y\le r\le z$.
      Therefore, we may replace $\sigma$ with $(\sigma_1,\dots,\sigma_y)$
      and deduce the claim from Proposition~\ref{ImageDistinct}. Note that
      if~$\S\in\SStd^\nu(\sigma,\eta)$ and $1\le r<y$ then
      $\S^{(r,y]}_r=\S^{(r,z)}_r$ since $\S^a_a=\sigma_a$ when
      $y\le a\le z$.
    \end{proof}

    \begin{claim}[3]
      Suppose that $\eta$ is a partition of $n$ such that $\eta\subset\nu$. 
      Then
      $$m_{\tnu_\eta}\LL\equiv q^C\sum_{\S\in\SStd^\nu(\mu,\sigma,\eta)} 
             \prod_{r=1}^{z-1}\Big( [\S_r^{(r,z]}]_\q^!
             \prod_{j=0}^{\gamma-\S_r^{(r,z]}-1} [c_z-c_r-j]_\q\Big)\, m_\S
                      \pmod{[c_y-c_z]S^\nu_{\calZ}}.$$
      where $C\in\Z$ and
      $\SStd^\nu(\mu,\sigma,\eta)
      =\set{\S\in\SStd^\nu(\mu,\eta)|\S^{>y}_r=0\text{ for }1\le r<y}$.
    \end{claim}

    \begin{proof}[Proof of Claim 3]
       Proposition~\ref{ImageDistinct} shows that  $m_{\tnu_\eta}\LL$
       is a linear combination of terms $m_\S$, for
       $\S\in\SStd^\nu(\mu,\eta)$ and, moreover, if
       $\S\in\SStd^\nu(\mu,\sigma,\eta)$ then the coefficient of $m_\S$
       is exactly as above. On the other hand, if
       $\S\in\SStd^\nu(\mu,\eta)\setminus\SStd^\nu(\mu,\sigma,\eta)$
       then $\S^{(y,z]}_y<\gamma$ so, by
       Proposition~\ref{ImageDistinct} again, the coefficient of
       $m_\S$ in $m_{\tnu_\eta}\LL$ is divisible by $[c_y-c_z]$. This
       proves the claim.
    \end{proof}

    Armed with these three claims we now return to the proof of
    Theorem~\ref{composition}. Combining Claims~1--3 shows that if
    $\t\in\Std(\nu)$ then, modulo~$[c_y-c_z]S^\nu_\calZ$,
    $m_\t\LL=m_\t\LLls\LLsm$ is equal to a linear combination of terms
    $m_\S$ where $\S\in\SStd^\nu(\mu,\sigma,\eta)$ where the coefficient 
    of~$m_\S$ is equal to the product of the coefficients coming from
    multiplication by $\LLls$ (Claim~2) and multiplication by $\LLsm$
    (Claim~1). (Furthermore, $C=C_1+C_2$.) The coefficients in Claim~1
    determine the polynomials $\beta_{\sigma\mu}(q)$, via
    Lemma~\ref{divides1}. Similarly, the coefficients in Claim~2
    determine the polynomials
    $\beta_{\la\sigma}(q)$ and those in Claim~3 determine
    $\beta_{\la\mu}(q)$. By Lemma~\ref{divides1} the polynomial
    $\beta_{\la\sigma}(q)\beta_{\sigma\mu}(q)$ divides all of the
    coefficients of the terms appearing in $m_{\tnu_\eta}\LL$
    according to Proposition~\ref{ImageDistinct}. Therefore, in the
    proof of Theorem~\ref{CPDistinctPrime} we can take
    $\beta_{\la\mu}(q)=\beta_{\la\sigma}(q)\beta_{\sigma\mu}(q)$.
    Note that, as in the proof of Theorem~\ref{CPDistinctPrime}, the
    terms in $[c_y-c_z]S^\nu_\calZ$ in Claim~3 do not contribute to the image of
    $\theta_{\la\mu}$ because $c_z -c_y \equiv 0 \pmod{ep^{\ell_p(\gamma^\ast)}}$. Therefore, 
    $\theta_{\la\mu} =  \theta_{\la\sigma}\theta_{\sigma\mu}$ as required.
\end{proof}

\begin{remark}
    The polynomials $\beta_{\la\mu}(q)\in F[q]$ are not uniquely
    determined by Lemma~\ref{divides1}. The proof of
    Theorem~\ref{composition} really shows that we can choose these
    polynomials so that, under the assumptions of the theorem,
    $\beta_{\la\mu}(q)=\beta_{\la\sigma}(q)\beta_{\sigma\mu}(q)$.
    Without this choice of $\beta$-polynomials, all we can say is that
    $\theta_{\la\mu} =u\theta_{\la\sigma}\theta_{\sigma\mu}$ for some
    non-zero scalar $u\in F$.
\end{remark}

\subsection{Jantzen filtrations} \label{JantzenFiltrations}
In this section we connect the Jantzen filtrations and the Carter-Payne
homomorphisms constructed in Section~\ref{ProofSection}.  If $p=0$ our result says
that the image is in contained in the radical of $S^\mu$, which is
automatically true, so this result is most interesting when $F$ is a field
of positive characteristic. The key to the proof is the observation
that if $q=\zeta$ then we can write~$\LL$ in two different ways using
the element $\LL'$ defined below.

The Hecke algebra $\h$ is defined over the field $F$ with parameter
$\zeta$. Let $\q$ be an indeterminate over~$F$ and let $\O=F[\q]_{(\q)}$ be
the localization of $F[\q]$ at the maximal ideal generated by~$\q$.  Then
$\O$ is a discrete valuation ring with maximal ideal $\pi=\q\O$, the
polynomials in $F[\q]$ with zero constant term. For $0\ne f\in\O$ define
$\val_\pi(f)=k$ where $k$ is maximal such that $f\in\pi^k$.  Let $K=F(\q)$ be
the field of fractions of~$\O$. We consider $F$ as an $\O$-module by
letting $\q$ act on $F$ as multiplication by~$\zeta$.

Let $\HO$ be the Hecke algebra of $\Sym_n$ over~$\O$ with (invertible)
parameter $\q+\zeta$. Then $\h\cong\HO\otimes_\O F$ and $\HK=\HO\otimes_\O
K$ is (split) semisimple. Thus $(K,\O,F)$ is a modular system, with
parameter $\q+\zeta$, for the algebras $(\HK,\HO,\h)$.

The algebra $\HO$ is cellular with cell modules the Specht modules
$S^\mu_\O$, for $\mu$ a partition of $n$. We have that
$S^\mu_K=S^\mu_\O\otimes_\O K$ is irreducible and
$S^\mu=S^\mu_F=S^\mu_\O\otimes_\O F$ is the $\h$-module defined in
section~2.1. As $\HO$ is cellular, the Specht module $S^\mu_\O$ comes
equipped with a bilinear form $\<\ ,\ \>_{\O,\mu}=\<\ ,\ \>_\mu$. For 
each positive integer~$i$ define 
$$J^i(S^\mu_\O) 
  = \set{x\in S^\mu_\O|\<x,y\>_{\mu}\in\pi^i\text{ for all }y\in S^\mu_\O}.$$
Finally, define $J^i(S^\mu)=(J^i(S^\mu_\O)+\pi J^i(S^\mu_\O))/J^i(S^\mu_\O)$, 
for $i\in\Z$. Then 
$$S^\mu=J^0(S^\mu)\supseteq J^1(S^\mu)\supseteq\dots$$
is the \textbf{Jantzen filtration} of $S^\mu$ relative to the
modular system $(K,\O,F)$. 

As in the last section, we assume that $\la$ and $\mu$ form a
separated $(e,p)$-Carter-Payne pair with parameters $(a,z,\gamma)$. We can again
assume that $a=1$, $z=\max\set{i|\la_i\ne0}$ and we define~$\nu$ to be
the partition of $n+\gamma$ obtained by adding $\gamma$ nodes to the
first row of $\la$.

As a slight variation on the definition of $\LL$ in section~2.2 set
$$L_{\la\mu}'=\prod_{i=1}^{z-1}\prod_{j=1}^{\gamma-1}(L_{n+j}-[c_i])
                       \cdot\prod_{i=2}^z(L_{n+\gamma}-[c_i]).$$
Since $\prod_{i=2}^z(L_{n+\gamma}-[c_i])$ divides $\LL'$, the
following result is easily proved using Lemma~\ref{onnj} below. 
We leave it as an exercise for the reader.

\begin{lemma}\label{killing}
    Suppose that $\t\in\Std(\nu)$ and that $\row_\t(n+\gamma)>1$. Then 
    $m_\t L_{\la\mu}'=0$.
\end{lemma}

As a consequence, if $M_1$ is the submodule of $S^\nu$ which appears
in the filtration of $S^\nu$ described in Lemma~\ref{restriction},
then $M_1L_{\la\mu}'=0$.

The Specht module $S^\nu_\O$ also carries an analogous inner product 
$\<\ ,\ \>_\nu$. The inner products $\<\ ,\ \>_\mu$ and
$\<\ ,\ \>_\nu$ are determined by the multiplication in $\h$ and
$\h[n+\gamma]$, respectively; see, for example, \cite[(2.8)]{M:ULect}.
These inner products are associative in the sense that
$\<xh,y\>_\nu=\<x,yh^*\>$, for all $x,y,\in S^\nu_\O$ and
$h\in\h[n+\gamma]^\O$, where $*$ is the unique anti-isomorphism of
$\h[n+\gamma]^\O$ such that $T_w^*=T_{w^{-1}}$ for all
$w\in\sym[n+\gamma]$. In particular, if $1\le k\le n+\gamma$ then
$\<xL_k,y\>_\nu=\<x,yL_k\>_\nu$, so that $\<x\LL,y\>_\nu=\<x,y\LL\>_\nu$, for
all $x,y,\in S^\nu_\O$. Proofs of all of these facts can be found in
\cite[Chapt.~2]{M:ULect}.

Since $\t^\nu_\mu=\tnu$ we have the following.

\begin{lemma}\label{inner product}
    Consider $S^\mu_\O$ as an $\h$-submodule of $S^\nu_\O$ as in
    Lemma~\ref{restriction}. Then
    $$\<x,y\>_\nu=\<x,y\>_\mu, \qquad\text{for all }x,y\in S^\mu_\O.$$
\end{lemma}

Recall that we defined the map $\val_{e,p}$ just before the statement
of Theorem~\ref{Jantzen} in the introduction and that (\ref{CP hom})
defines a polynomial $\beta_{\la\mu}(q)\in F[q]$
whenever~$\la$ and~$\mu$ form a Carter-Payne pair.

We can now prove Theorem~\ref{Jantzen} from the introduction.

\begin{proof}[Proof of Theorem~\ref{Jantzen}]
    We have to show that the image of $\theta_{\la\mu}$ is contained
    in $J^{\delta}(S^\mu)$, where
    $\delta=\val_{e,p}(\la_a-\la_z+z-a+\gamma)-\val_{e,p}(\gamma)$. To
    do this we work in $\h[n+\gamma]^\O$.
    Let $\LL^\O$ and $\LLOp$ be the elements
    of $\HO$ which are obtained from $\LL$ and $\LL'$, respectively,
    by replacing $q$ with $\q+\zeta$. Using the simple identity
    $[c_1]_{q+\zeta}=[c_z]_{q+\zeta}+q^{c_z}[c_1-c_z]_{q+\zeta}$, we see
    that
    $$\LLO=\prod_{i=1}^{z-1}\prod_{j=1}^\gamma(L_{n+j}-[c_i]_{\q+\zeta})
	        =\LLOp - q^{c_z}[c_1-c_z]_{\q+\zeta}\LLOpp,
    $$
where $\LLOpp=\prod_{i=2}^{z-1}\prod_{j=1}^{\gamma}(L_{n+j}-[c_i]_{\q+\zeta})
               \cdot\prod_{j=1}^{\gamma-1}(L_{n+j}-[c_1]_{q+\zeta})$.
Therefore, when we specialize at $q=0$,
$$\LL=\LL^\O\otimes_\O1=\LLOp\otimes_\O1=\LL'$$ 
in $\h$ since $c_1\equiv c_z\pmod e$. So
multiplication by $\LL$ and $\LL'$ induce the same $\h$-homomorphism
$S^\la\to S^\mu$, which may be zero, by the argument of
Theorem~\ref{CPDistinctZero}.

In the proof of Theorem~\ref{CPDistinctPrime}, the homomorphism
$\theta_{\la\mu}$ was defined to be the specialization of the map
$m_\t\mapsto\frac1{\beta_{\la\mu}(\q+\zeta)}m_\t\LLO$ at $q=0$, for
$\t\in\Std_\la(\nu)$. Set
$h=\la_a-\la_z+z-a+\gamma=c_1-c_z$, so that $\delta=\val_{e,p}(h)$. By
assumption, if $l=\ell_p(\gamma^*)$ then $h\equiv0\pmod{ep^l}$. If we write
$h=h'ep^l$, for some $h'\in\Z$, then
$$[c_1-c_z]_{\q+\zeta}=[h'ep^l]_{\q+\zeta}= [ep^l]_{\q+\zeta}
[h']_{(\q+\zeta)^{p^l}} =[e]^{p^l}[h']_{(\q+\zeta)^{p^l}}.$$ Hence,
$\val_\pi([h]_{\q+\zeta})\ge p^l=\val_{e,p}(h)=\delta$.
                           
Recall that $\LLO=\LLOp+[c_1-c_z]_{\q+\zeta}\LLOpp$. Suppose that
$\t\in\Std_\la(\nu)$. By Lemma~\ref{inner product}, if $x$ belongs
to~$S^\mu_\O$ then
\begin{align*}
\<m_\t\LLO,x\>_\mu=\<m_\t\LLO,x\>_\nu 
         &=\<m_\t\LLOp,x\>_\nu-q^{c_z}[h]_{\q+\zeta}\<m_\t\LLOpp,x\>_\nu\\
         &=-q^{c_z}[h]_{\q+\zeta}\<m_\t\LLOpp,x\>_\nu,
\end{align*}
where the last equality follows because
$\<m_\t\LLOp,x\>_\nu=\<m_\t,x\LLOp\>_\nu=0$ by 
Lemma~\ref{killing}. 

If $\gamma<e$ then $\beta_{\la\mu}(q+\zeta)=1$ and the proof is complete. If
$\gamma\ge e$ it remains to account for dividing by
$\beta_{\la\mu}(q+\zeta)$ in the definition of $\theta_{\la\mu}$. Observe that 
if $x\in S^\mu_\O$ then $x$ is a linear combination of terms~$m_\s$ with
$\s\in\Std_\mu(\nu)$. If $\s\in\Std_\mu(\nu)$ then $\row_\s(n+j)=z$, for
$1\le j\le\gamma$.  Therefore, $m_\s L_{n+j}=[c_z-j+1]m_\s$ by
Lemma~\ref{onnj} below, for example, so that
\begin{align*}
\<m_\t\LLOpp,x\>_\nu &= \<m_\t,x\LLOpp\>_\nu\\
    &=\prod_{i=2}^{z-1}\prod_{j=0}^{\gamma-1} q^{c_i}[c_z-c_i-j]_{q+\zeta}
    \cdot\prod_{j=0}^{\gamma-2}q^{c_1}[c_z-c_1-j]_{q+\zeta}
     \cdot\<m_\t,x\>_\nu.
\end{align*}
Let $\beta'_{\la\mu}(q+\zeta)$ be the coefficient of $\<m_\t,x\>$  in the
last equation.  Recall from the proof of Theorem~\ref{CPDistinctPrime} that the
polynomial $\beta_{\la\mu}(q+\zeta)$ is a product of
$z-1$ factors corresponding to the row index $i=1,2,\dots,z-1$ above.
Noting that $c_1\equiv c_z\pmod e$, we have that
$$\val_\pi\([\gamma]_{q+\zeta}\beta'_{\la\mu}(q+\zeta)\)
        \ge\val_\pi\(\beta_{\la\mu}(q+\zeta)\)$$
by taking $X=0$ in Corollary~\ref{CMgamma3}. This completes the proof.
\end{proof}

It would be interesting to know how tight the bound obtained in
Theorem~\ref{Jantzen} is. That is, to determine the maximal $\delta'$ such
that the image of $\theta_{\la\mu}$ is contained in $J^{\delta'}(S^\mu)$.

If $\gamma<e$ then $\beta_{\la\mu}(q)=1$. Hence, as a special case of the
Theorem we obtain the following. 

\begin{corollary}
  Suppose that $p>0$, $\gamma<e$ and that $\lambda$ and $\mu$ form an
  $(e,p)$-Carter-Payne pair with parameters $(a,z,\gamma)$ such that
  $\lambda_r-\lambda_{r+1}\ge\gamma$, whenever $a\le r\le z$.  Then
  $\im\theta_{\la\mu}\subseteq J^{\delta}(S^\mu)$, where
  $\delta=\val_{e,p}(\la_a-\la_z+z-a+\gamma)$.
\end{corollary}

When $\zeta=1$ and $\gamma=1$ this result has already been proved by Ellers
and Murray~\cite[Theorem~7.1]{EM} without assuming that
$\lambda_r-\lambda_{r+1}\ge\gamma$, for $a\le r\le z$. The proof of
Theorem~\ref{Jantzen} was inspired by the argument of Ellers
and Murray.

We note that when $\zeta=1$ we can replace the modular system
$(K,\O,F)$ used above with $(\Q_{(p)}, \Z_{(p)}, \Z/p\Z)$ and the
valuation map $\val_\pi$ with the usual $p$-adic valuation map
$\val_p$. With these choices, we obtain the `natural' Jantzen
filtration of $S^\mu$ and the argument above shows that we can take
$\delta=\val_p(c_1-c_z)-\val_p(\gamma)$.

\subsection{The $(e,p)$-Carter-Payne Theorem}
The techniques used in this paper to prove Theorem~\ref{CPDistinctZero} and
Theorem~\ref{CPDistinctPrime} can be used to prove the existence of
homomorphisms between other pairs of Specht modules. As we now sketch, it
is likely that a complete proof of Theorem~\ref{CarterPayne} could be given
using these ideas.

Fix a pair of partitions $\la$ and $\mu$ of $n$ which form a Carter-Payne
pair with parameters $(a,z,\gamma)$. As in the last section we may assume
that $a=1$ and that $z$ is the length of $\la$. Let $\nu$ be the partition 
of $n+\gamma$ given by 
\[\nu_r = \begin{cases}
\la_r+\gamma, & r=1, \\
\la_r, & \text{otherwise}.
\end{cases}\]
Write $\nu=(\nu_1^{b_1},\nu_2^{b_2},\ldots,\nu_s^{b_s})$ where $\nu_1>\nu_2>\ldots>\nu_s>0$, and  set $B_i=\sum_{k=1}^i b_k$ for $1\le i\le s$. Then the nodes that can be removed from $\Diag(\nu)$ to leave the diagram of a partition are at the ends of the rows $B_1,B_2,\ldots,B_s$. 
Set $c_r = \nu_r -B_r$, for $1\le r\le s$, so that  $c_r$ is the content of the $r^{\text{th}}$ removable node of $\nu$: 

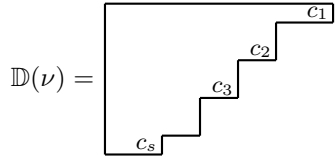
\begin{center}
\begin{pspicture}(-3.75,0)(6,2.25)
\psset{xunit=0.25cm, yunit=0.25cm}
\psline(0,0)(0,8)
\psline(0,8)(12,8)
\psline(12,8)(12,7)
\psline(12,7)(9,7)
\psline(9,7)(9,5)
\psline(9,5)(7,5)
\psline(7,5)(7,3)
\psline(7,3)(5,3)
\psline(5,3)(5,1)
\psline(5,1)(3,1)
\psline(3,1)(3,0)
\psline(3,0)(0,0)
\rput(-2.75,4){$\Diag(\nu)=$}
\rput(11.2,7.5){\small$c_1$}
\rput(8.2,5.5){\small$c_2$}
\rput(6.2,3.5){\small$c_3$}
\rput(2.2,0.5){\small$c_s$}
\end{pspicture}
\end{center}

Now define
\[\LL=\prod_{r=1}^{s-1}\prod_{j=1}^{\gamma}\Big(L_{n+j}-[c_r]\Big).\]

Arguing as in the proof of Theorem~\ref{CPDistinctZero} or
Theorem~\ref{CPDistinctPrime} it is possible to show that right
multiplication by $\LL$ induces a $\h$-homomorphism $S^\la\to
S^\mu$.  However, it is not clear that this homomorphism is non-zero.  

If $\la$ and $\mu$ form an $(e,p)$-Carter-Payne pair with parameters
$(a,z,\gamma)$ where $\gamma=1$ then using Corollary~\ref{CONMJ} below, or by
following Ellers and Murray~\cite{EllersMurray}, it is possible to show
that right multiplication by $\LL$ induces a non-zero $\h$-homomorphism
from $S^\la$ to $S^\mu$.  

\begin{conjecture}
Suppose that $\gamma <e$. Then right multiplication by $\LL$ induces a
non-zero $\h$-homomorphism from $S^\la$ to $S^\mu$.  
\end{conjecture}

By the argument used to prove Theorem~\ref{Jantzen}, if this
conjecture is true then the image of this homomorphism is contained in
$J^{\delta}(S^\mu)$, where $\delta=\val_{e,p}(\la_a-\la_z+z-a+\gamma)$.

We end with two examples.  

\begin{Example}
Suppose that $\la=(4,4,3,2)$, that $\mu=(6,4,3)$ and that $e=7$.  
If we take $\t=\tnu_\la$ and $\LL=(L_{15}-[5])(L_{15}-[2])L_{15}(L_{14}-[5])(L_{14}-[2])L_{14}$ then direct computation shows that  
\begin{align*}
m_{\t} \LL = & 
q^{-5}(q^3-q-1)[2][2][4]\; {\tab(111122,2233,344,56)}
-q^{-5|}[2][2][2][4] \;{\tab(111122,2234,334,56)} \\
&\hspace*{10mm}-q^{-4}[2][2][4]\; {\tab(111123,2223,344,56)}
+q^{-6}[2][2][4] \; {\tab(111123,2224,334,56)} \\
&\hspace*{10mm}+q^{-6}[2][2][4] \;{\tab(111124,2223,334,56)} 
-q^{-9}[2][2][3][4] \; {\tab(111124,2224,333,56)} \\
&\hspace*{10mm}+q^{-9}[2][2][4][5]\; {\tab(111133,2222,344,56)}  
-q^{-11}[2][2][4][5]\; {\tab(111134,2222,334,56)} \\
&\hspace*{10mm}+q^{-4}[2][2][3][4][5]\; {\tab(111144,2222,333,56)}.
\end{align*}
Further, if $\t=\tnu_\eta$ for some $\nu\ne \eta$ then $m_{\t}\LL$ has
a factor of $[7]$.  Thus if $e=7$ (and $p$ is arbitrary) there exists a non-zero homomorphism
$\theta:S^\la \rightarrow S^\mu$.

Note that the coefficient of the first tableau is not a product of Gaussian
polynomials multiplied by a power of $q$.  This indicates that the polynomial
coefficients appearing in a general version of Proposition
\ref{ImageDistinct} may be difficult to describe.  
\end{Example}

\begin{Example}
Finally let us consider the case that $\la=(4,3,3)$ and $\mu=(7,3)$.  If we take $\t=\tnu_\la$, and $\LL=(L_{10}-[6])(L_9-[6])(L_8-[6])$ then direct computation shows that  
\begin{align*}
m_{\t}\LL = &-q^6[2][3]{\tab(1111222,333,444)} 
               + q^5[2]{\tab(1111223,233,444)}\\
            &\hspace*{10mm}-q^3[2]{\tab(1111233,223,444)}+[2][2]{\tab(1111333,222,444)}.
\end{align*}
If $\t=\tnu_\eta$ for some $\nu\ne \eta$ then $m_{\t}\LL$ has a factor of
$[6]$.  So if $e=2$ and $p=3$ then (after dividing by $[2]$), we have shown
that there is a non-zero homomorphism between $S^\la$ and $S^\mu$, as
predicted by the Carter-Payne theorem.  However, we have shown that
if~$e=3$ and $p$ is arbitrary then we there is a non-zero homomorphism.
These maps are not Carter-Payne homomorphism except when $p=2$, although
they are described by Parker~\cite{Parker}.     

It is interesting to note that in ~\cite{FayersMartin:homs} the authors
show the existence of such a homomorphism in the case when $e=p=3$; that
is, when $\h=F_3 \sym[10]$.
\end{Example}

\section{Jucys-Murphy elements acting on almost initial tableaux}
In this section we complete the proof of our main results in
Sections~\ref{ProofSection}--\ref{JantzenFiltrations} by proving some very
precise formulas which describe how the Jucys-Murphy elements act on
certain elements of the Specht modules.  The results in this section are
valid for an arbitrary Hecke algebra $\h[n+\gamma]=\h[n+\gamma]^F$ defined
over an ring~$F$ with invertible parameter~$q$. Nonetheless, throughout we
work with the generic Hecke algebra $\h[n+\gamma]^\calZ$ as we prefer to
think of $[k]=[k]_q$ as a polynomial in~$q$.  The results in this section
are independent of the results in
Sections~\ref{ProofSection}--\ref{JantzenFiltrations}.

Throughout this section we fix integers $n,\gamma >0$ and an
\textit{arbitrary} partition $\nu$ of $n+\gamma$. (In this section
the only result which requires the assumption that
$\nu_i-\nu_{i+1}\ge\gamma$, for $1\le i<z$, is
Proposition~\ref{Image}.) Let
$z=\max\set{r |\nu_r>0}$.  Recall that $\{T_w \mid w \in
\Sym_{n+\gamma}\}$ is a basis of~$\h[n+\gamma]^\calZ$.

\subsection{Semistandard basis elements}
We now fix notation that will be used extensively for the rest of the
paper. Suppose that~$i$ and~$j$ are integers such that 
$1\le i\le j\le n+\gamma$. Define
\[T_{i,j} = \prod_{l=i}^{j-1}T_l\]
and for $i<k \le j$ define
\[T_{i,j\setminus k}= \prod_{l=i}^{k-2}T_l \cdot \prod_{l=k}^{j-1} T_l.\]
Our convention will always be to read products from left to right, so that
$$T_{i,j}=T_iT_{i+1}\dots T_{j-1} \quad \text{ and } \quad T_{i,j\setminus k}= T_i T_{i+1}\dots T_{k-2}T_{k}\dots T_{j-1}.$$
In particular, $T_{i,i}=1$, $T_i=T_{i,i+1}$, $T_{i+1,j}=T_{i,j\setminus i+1}$ and $T_{i,j-1}=T_{i,j\setminus j}$.
Recall that for $1 \leq k \leq n+\gamma$ we defined the Jucys-Murphy element
$L_k$.  
Similarly, we set
$$L'_k=q^{1-k}T_{k-1}\dots T_1T_{1,k},\qquad\text{ for }1\le k\le n.$$
The reader can check that $L'_k=(q-1)L_k+1$. Consequently, the elements
$L_k$ and $L'_k$ are almost interchangeable.

Let $S^\nu$ be the 
$\h[n+\gamma]^\calZ$-module corresponding to the partition $\nu$, so that
$S^\nu$ has basis $\{m_\t \mid \t \in \Std(\nu)\}$. 
If $\s \in \rowstd(\nu)$ and $1\le k\le n$ then the
\textbf{content} of $k$ in $\s$ is $\res_\s(k)=c-r$, if $\s(r,c)=k$.

\begin{lemma}\label{permutation}
Suppose that $1\le i\le n+\gamma-1$ and that $\s \in \rowstd(\nu)$.  Then
\[m_\s T_i = \begin{cases}
m_{\s (i,i+1)}, &  i \text{ lies above } i+1 \text{ in } \s, \\
qm_\s, & i \text{ and } i+1 \text{ lie in the same row of } \s, \\
qm_{\s (i,i+1)}+ (q-1)m_{\s}, & \text{otherwise}.
\end{cases}\]
\end{lemma}

\noindent Note that if $\s$ is standard then the tableau $\s(i,i+1)$ is
also standard unless $i$ and $i+1$ are in the same column.

\begin{proof}
The result holds for the row standard basis, 
$\set{m_\nu T_{d(\s)}|\s\in\rowstd(\nu)}$, of the permutation module
$M^\nu=m_\nu\h[n+\gamma]^\calZ$ by~\cite[Corollary 3.4]{M:ULect}. As $m_\s$ is just
the image of $m_\nu T_{d(\s)}$ under the natural projection map
$M^\nu\to S^\nu$ the result follows.
\end{proof}

\begin{lemma} \label{initialmurphy}
Suppose that $1\le k\le n$.  Then 
\[m_{\tnu}L_k = [{\rm c}_{\tnu}(k)] m_{\tnu} \quad\text{and}\quad
    m_{\tnu} L'_k= q^{\res_{\tnu}(k)}m_{\tnu}.  \]
\end{lemma}

\begin{proof}
The first identity follows from~\cite[Theorem 3.32]{M:ULect}. The
second identity follows from the first using the fact that
$L_k'=(q-1)L_k+1$.
\end{proof}

\begin{lemma} \label{Murphyformulae} \quad
Suppose that $1\le i\le i'\le n+\gamma-1$ and $1\le j,j'\le n+\gamma$.  Then
\begin{enumerate} 
\item $L_j L_{j'} = L_{j'} L_j$,
\item $T_i L_j = L_j T_i$ if $i\ne j,j-1$,
\item $T_i L_i = L_{i+1}T_i -L'_{i+1}$,
\item $T_i L_{i+1}=L'_{i+1}+L_iT_i$,
\item $T_i (L_i+L_{i+1})=(L_i+L_{i+1})T_i$,
\item $T_i L_i L_{i+1}= L_i L_{i+1} T_i$,
\item $T_{i,i'}L_{i'}=L_iT_{i,i'}+\sum_{x=i+1}^{i'} L'_xT_{i,i'\setminus x}$.  
\end{enumerate}
\end{lemma}

\begin{proof}
All but the last identity are given 
in~\cite[Proposition 3.26 and Exercise 3.6]{M:ULect}. Part~(g) is readily 
proved by induction on $i'-i$.    
\end{proof}

Suppose that $\alpha$ is a partition and that $\beta$ is a composition of an integer $m$ 
and let $\S$ be an $\alpha$-tableau of type $\beta$.
Recall from Section~\ref{ProofSection} that 
\[m_\S =\sum_{\substack{\s \in \rowstd(\alpha) \\ \beta(\s)=\S}} m_\s.\] 
By definition $m_\S\in S^\alpha$. We need a different description of $m_\S$.

Define $\dot\S$ to be the unique row standard tableau such that
$\beta(\dot\S)=\S$ and the numbers in each row of $\t^\beta$ appear in
row order in $\dot\S$. Then $d(\dot\S)$ is the unique element of
minimal length in the double coset $\Sym_\alpha d(\dot\S)\Sym_\beta$
by \cite[Prop.~4.4]{M:ULect}, and by \cite[(4.6)]{M:ULect} $$m_\S=
m_{\dot\S} \sum_{w\in\D_\S} T_w,$$ where $\D_\S$ is the set of all $w
\in \Sym_\beta$ such that if $i<j$ lie in the same row of $\dot\S w$
then $(i)w<(j)w$. In fact, by~\cite[Prop.~4.4]{M:ULect} again,
$\D_\S=\D_\sigma\cap\Sym_\beta$ where the composition~$\sigma$ is
given by $\Sym_\sigma=d(\dot\S)^{-1}\Sym_\alpha d(\dot\S)\cap
\Sym_\beta$ and $\D_\sigma=\set{d(\s)|\s\in\rowstd(\sigma)}$ is the
set of distinguished (or minimal length) right coset representatives
of $\Sym_\sigma$ in $\sym$. Write $\beta=(\beta_1,\dots,\beta_b)$.
Then $\Sym_\beta=\Sym_{\beta_1}\times\dots\times\Sym_{\beta_b}$ and every
element $w$ of $\Sym_\beta$ can be written uniquely as a product of
commuting permutations $w=w_1\dots w_b$ where, abusing notation
slightly, $w_i\in\Sym_{\beta_i}$ for $1\le i\le b$. Let
$\D_\S(i)=\D_\S \cap \Sym_{\beta_i}$ for $1\le i\le b$.  Define
$D_\S=D_\S(1)\dots D_\S(b)$, where $D_\S(i)=\sum_{w\in\D_\S(i)} T_w$.
Then we have \begin{equation}\label{m_S expansion} m_\S = m_{\dot\S}
    D_\S = m_{\t^\alpha} T_{d(\dot\S)}D_\S.  \end{equation}

\begin{Example}
Suppose that  $\alpha=(7,2)$, $\beta=(4,3,2)$ and 
$\S={\tab(1112233,12)}\in\rowstd(\alpha,\beta)$. Then
$\dot\S={\tab(1235689,47)}$ and
\[m_\S= m_{\t^\alpha} T_{7,8}T_{6,7}T_5T_4(1+T_3+T_3T_2+T_3T_2T_1)(1+T_6+T_6T_5).\]
\end{Example}

\begin{lemma} \label{CosetMix}
Let $a,b,c$ and $g$ are integers with $1\le a < c < b\le m$ and
$g\notin\{a,\dots,b\}$ and let $\beta=(1^{a-1},b-a+1,1^{m-b})$, a composition
of $m$. Suppose $\alpha$ is a partition of $m$ and that $\t$ is a row-standard $\alpha$-tableau such that
$a,\ldots,b$ are in row order in~$\t$, $\row_\t(c-1)<\row_\t(c)$ and
$i'=\row_\t(g)<i=\row_\t(c)$.  Let $\s=\t(c,g)$, $\T=\beta(\t)$ and
$\S=\beta(\s)$. Then
$$m_\s\Big(\sum_{j=c}^{c+l}T_{c,j}\Big)D_\T(a)
      =q^s [\S^{a}_{i'}]\, m_\S,$$
where $l=\S^{a}_i$ and $s=\S^{a}_{(i',i)}$. 
\end{lemma}

\begin{proof} We prove the Lemma using some standard properties of the
    distinguished coset representatives of Coxeter groups. To exploit
    these results it is convenient to introduce some new notation.

    If $\sigma$ is a composition of $m$ let $J_\sigma=\set{1\le
    i<m|\row_{\t^\sigma}(i)=\row_{\t^\sigma}(i+1)}$. Then
    $\Sym_\sigma$ is generated by $\set{(i,i+1)|i\in J_\sigma}$ and  the 
    map $\sigma\mapsto J_\sigma$ defines a bijection between the set of 
    compositions of $m$ and the subsets of $\Pi_m=\{1,2,\dots,m-1\}$. If
    $J=J_\sigma\subseteq\Pi_m$ set $m_J=m_\sigma$, $\Sym_J=\Sym_\sigma$,
    $\D_J=\D_\sigma$ and $D_J=D_\sigma$. If $J\subseteq K\subseteq\Pi_m$ set
    $\D_J^K=\D_J\cap\Sym_K$. Then $\D_J^K$ is a complete set of coset
    representatives for $\Sym_J$ in $\Sym_K$ and, moreover, the following
    two properties hold:
    \begin{enumerate}
	    \item[(D1)] Suppose that $J\subseteq K\subseteq A\subseteq\Pi_m$. 
	                Then $D_J^A=D_J^KD_K^A$.
	    \item[(D2)] Suppose that $J,K,L\subseteq\Pi_m$ with $J\subseteq K$
	      and $|k-l|>1$ for all $k\in K$ and $l\in L$. Then
	      $D_J^K=D_{J\cup L}^{K\cup L}$.
    \end{enumerate}
    Property (D1) is well-known and easy to prove: see, for
    example, \cite[Lemma~2.1]{BBHT}. The second statement (D2) is trivial
    because the assumptions imply that 
    $\Sym_{K\cup L}=\Sym_K\times\Sym_L$ and 
    $\Sym_{J\cup L}=\Sym_J\times\Sym_L$. 

    Let $A=\{a,a+1,\dots,b-1\}$ and let 
    $E=\set{e\in A|\row_\t(e)=\row_\t(e+1)}$. Then $\D_\T(a)=\D_\T=\D_{E}^A$.
    Similarly, let
    \begin{align*}
      E'&=\set{e\in A|\row_{\dot\S}(e)=\row_{\dot\S}(e+1)}\\
       &=\set{e\in E|\row_\t(e)\notin(i',i)}\cup
         \set{e+1|e\in E\text{ and }\row_\t(e)\in[i',i)}
         \setminus\{c\},
    \end{align*}
    Then $\D_\S(a)=\D_S=\D^A_{E'}$.  To prove the Lemma we consider
    various subsets of~$A$ which depend on $E$ and $E'$. Let 
    $$C=\set{e\in E\cap E'|\row_\t(e)=i} \quad\text{and}\quad
           C'=\set{e\in E\cap E'|\row_\t(e)=i'}$$
    and let $L,L'\subseteq A$ be the subsets of ~$A$ such that
    $$E= C\sqcup\{c\}\sqcup L\qquad\text{and}\qquad
      E'=C'\sqcup\{c'\}\sqcup L',\qquad\text{(disjoint unions)},
    $$
    where $c'\in A$ is maximal such that $\row_\t(c')=i'$. Note that 
    $c'\le c$ and $\S^{a}_{(i',i)}=c-c'$. In particular, $c=c'$ if and
    only if $s=\S^a_{(i',i)}=0$.
    
    Armed with these definitions we can now prove the lemma. We have
    \begin{align*}
    m_\s\Big(\sum_{j=c}^{c+l}T_{c,j}\Big)D_\T(a)
    &= m_{\dot\S}T_{c',c}D^{C\cup\{c\}}_C D^A_E
     = m_{\dot\S}T_{c',c}D^E_{C\cup L} D^A_E\\
    &= m_{\dot\S}T_{c',c}D^A_{C\cup L}
    \intertext{where the last two equalities follow by (D2) and (D1),
    respectively. 
    Let $d=(c',c'+1)\ldots(c-1,c)$ so that $T_{c',c}=T_d$.  
    Then $\Sym_{C\cup L}=d^{-1}\Sym_{C'\cup L'}d$ and 
    $d\in\D_{C'\cup L'}\cap\D_{C\cup L}^{-1}$ so
    that $m_{C'\cup L'}T_d=T_dm_{C\cup L}$. (In fact, 
    $\D^A_{C'\cup L'}=d\D^A_{C\cup L}$ by \cite[Lemma~2.4]{BBHT}, however,
    this is not enough for our purposes because , in general,
    $D^A_{C'\cup L'}\ne T_dD^A_{C\cup L}$.) Now,
    $m_{\dot\S}T_w=q^{\ell(w)}m_{\dot\S}$ for all
    $w\in\Sym_{C'\cup L'}$, so $m_{\dot\S}=hm_{C'\cup L'}$ for
    some $h\in\h[m]^\calZ$. Consequently, continuing the last displayed equation,}
    m_\s\Big(\sum_{j=c}^{c+l}T_{c,j}\Big)D_\T(a)
      &=hm_{C'\cup L'}T_d D^A_{C\cup L}
       =hT_dm_{C\cup L}D^A_{C\cup L}\\
      &=hT_dm_A=q^{\ell(d)}hm_A
       =q^{\ell(d)}hm_{C'\cup L'} D_{C'\cup L'}^A.
    \intertext{Observe that $\ell(d)=c-c'=\S^a_{(i',i)}=s$. Therefore, 
    using (D1) and (D2) again,}
    m_\s\Big(\sum_{j=c}^{c+l}T_{c,j}\Big)D_\T(a)
       &=q^sm_{\dot\S}D^{E'}_{C'\cup L'}D^A_{E'}
        =q^sm_{\dot\S}D^{C'\cup\{c'\}}_{C'}D^A_{E'}\\
       &=q^s[\S_{i'}^a]m_{\dot\S}D^A_{E'}
    \end{align*}
    where the last equality follows because
    $m_{\dot\S}T_w=q^{\ell(w)}m_{\dot\S}$ for all
    $w\in\Sym_{C'\cup\{c'\}}$ by Lemma~\ref{permutation}. We have already
    observed that $\D_\S=\D^A_{E'}$, so an application of (\ref{m_S
    expansion}) now completes the proof.
\end{proof}

\begin{Example}
Suppose that $a=4$, $b=9$, $c=8$ and that $g=3$. Then
\[\t={\tab(1345,267,89)} \quad\implies 
\quad \s={\tab(1458,267,39)}, \quad\T={\tab(1344,244,44)}
           \quad\text{and}\quad  \S={\tab(1444,244,34)}.\]
Abusing notation and identifying $m_\s$ with $\s$ and $\S$ with $m_\S$, we
have
\[{\tab(1458,267,39)} \;(1+T_8) D_\T(4)=q^2 [3]\,{\tab(1444,244,34)}\] 
where $D_\T(4)= \sum_{w\in\D_\T} T_w$. By definition, $\D_\T(4)=\D_\T$ is
the set of minimal length coset representatives of 
$\Sym_{\{4,5\}}\times \Sym_{\{6,7\}} \times \Sym_{\{8,9\}}$ in 
$\Sym_{\{4,\ldots,9\}}$.  
\end{Example}

For any composition
$\sigma=(\sigma_1,\sigma_2,\dots)$ let
$\bar\sigma_{k}=\sigma_1+\dots+\sigma_k$, for $k\ge0$.

\begin{lemma} \label{initialperm}
Suppose that $\eta\subseteq\nu$ is a partition of $n$ and set
$\xi=(\nu_1-\eta_1,\nu_2-\eta_2,\dots)$, a composition of~$\gamma$.  Then
\[
m_{\tnu_\eta} = m_{\tnu} \prod_{i=0}^{z-1} 
\prod_{k=0}^{\xi_i-1} T_{\bar\nu_{z-i}-k,n+\bar\xi_{z-i}-k}.\] 
\end{lemma}

\begin{proof}
For $0\le j\le \gamma$, let $\t(j)$ be the $\nu$-tableau such that the entries $n+j+1,\ldots,n+\gamma$ appear in the same position that they appear in $\tnu_\eta$ and the entries $1,2,\ldots,n+j$ are in row order. 
Consider $\t(\gamma-1)$.  Suppose that $n+\gamma$ appears (at the end of) row $r$ in $\tnu_\eta$.  Then $m_{\t(\gamma-1)} = m_{\tnu}T_{\bar{\nu_r}}\ldots T_{n+\gamma-1}=m_{\tnu}T_{\bar{\nu_r},n+\gamma}$ by Lemma~\ref{permutation}.
The general case now follows by downwards induction on $j$ using essentially the same observations.   
\end{proof}

Similarly, it is straightforward to check the following lemma.

\begin{lemma} \label{ArbPerm}
    Suppose $\t \in \rowstd(\nu)$ and let $\eta=\Shape(\t_{\downarrow n})$.  Then
\[ m_\t = m_{\tnu} T_{d(\tnu_\eta)}T_w\]
for a unique permutation $w \in \sym \times \sym[\gamma]$.  
\end{lemma}

We are now ready to start proving the main results of this section.
Recall that if $\eta \subseteq \nu$ is a partition of $n$ then the
almost initial tableau $\tnu_\eta$ was defined in
Section~\ref{ProofSection}. If $1\le r\le z$ then define
$c^\eta_r=\eta_r-r$.  

\begin{lemma} \label{onij}
Suppose that $\t=\tnu_\eta$ is an almost initial tableau such that
$\row_{\t}(n+1)\ne z$ and let $j\ge 1$ be maximal such that
$r=\row_{\t}(n+j)<z$.  For $i\ge 1$ set $\xi_i=\nu_i-\eta_i$ and if
$1\le g\le n$ then let $c(g) = c^\eta_m$ where $\row_\t(g)=m$.  Then 
\[m_{\t}L_{n+j} =  [\res_{\t}(n+j)] m_{\t} 
        + q^{\xi_r-1} \sum_{g=\bar\nu_r-j+1}^n q^{c(g)}m_{\t(g,n+j)}.\]
\end{lemma}

\begin{proof} Using in turn, Lemma~\ref{initialperm}, 
  Lemma~\ref{Murphyformulae}(g) and Lemma \ref{initialmurphy}, we find
\begin{align*}
m_{\t}L_{n+j} &= \Big(m_{\tnu} \prod_{i=0}^{r-1} \prod_{k=0}^{\xi_{r-i}-1}T_{\bar\nu_{r-i}-k,n+\bar\xi_{r-i}-k}\Big) L_{n+j} \\
&= m_{\tnu}T_{\bar\nu_r,n+j}L_{n+j}\Big( \prod_{k=1}^{\xi_r-1}T_{\bar\nu_{r}-k,n+j-k} \Big)
   \Big(\prod_{i=1}^{r-1}\prod_{k=0}^{\xi_{r-i}-1}T_{\bar\nu_{r-i}-k,n+\bar\xi_{r-i}-k}\Big)\\
&=  m_{\tnu} \Big(L_{\bar\nu_r}T_{\bar\nu_r,n+j}+\sum_{x=\bar\nu_r+1}^{n+j} L'_x T_{\bar\nu_r,n+j 
   \setminus x} \Big) \Big(\prod_{k=1}^{\xi_r-1}T_{\bar\nu_{r}-k,n+j-k}\Big)\\
&\hspace*{10mm}\times \Big( \prod_{i=1}^{r-1} 
   \prod_{k=0}^{\xi_{r-i}-1}T_{\bar\nu_{r-i}-k,n+\bar\xi_{r-i}-k}\Big)\\
& = [\res_{\t}(n+j)] m_{\t} 
  +  m_{\tnu}\sum_{x=\bar\nu_r+1}^{n+j} q^{\res_{\tnu}(x)} T_{\bar\nu_r,n+j \setminus x} 
     \Big(\prod_{k=1}^{\xi_r-1}T_{\bar\nu_{r}-k,n+j-k}\Big)\\
& \hspace*{10mm} \times \Big( \prod_{i=1}^{r-1} \prod_{k=0}^{\xi_{r-i}-1}T_{\bar\nu_{r-i}-k,n+\bar\xi_{r-i}-k}\Big).
\end{align*}
Now fix $x$ with $\bar\nu_r+1\le x\le n+j$.  To complete the proof, 
we show that
\begin{align*}
q^{\res_{\tnu}(x)}m_{\tnu} T_{\bar\nu_r,n+j \setminus x}&
  \Big(\prod_{k=1}^{\xi_r-1}T_{\bar\nu_{r}-k,n+j-k}\Big)\Big( \prod_{i=1}^{r-1} 
  \prod_{k=0}^{\xi_{r-i}-1}T_{\bar\nu_{r-i}-k,n+\bar\xi_{r-i}-k}\Big)\\
  &=q^{\xi_r-1}q^{c(x-j)}m_{\t(x-j,n+j)}.
\end{align*}
Note that $x$ lies in the same position of $\tnu$ that $x-j$ lies in~$\t$. Let $\row_{\tnu}(x)=m$.  Therefore
\begin{align*}
m_{\tnu} T_{\bar\nu_r,n+j\setminus x}&= m_{\t^{\nu}(x-1,x-2,\ldots,\bar\nu_r)}T_{x,n+j} \\
&= q^{c(x-j)-\res_{\tnu}(x)} m_{\t^{\nu}(x-1,x-2,\ldots,\bar\nu_r)}T_{\bar\nu_m,n+j} \\ 
&=q^{c(x-j)-\res_{\tnu}(x)} m_{\t'}
\end{align*}
where $\t'=\tnu(x-1,x-2,\ldots,\bar\nu_r)(n+j,n+j-1,\ldots,\bar\nu_m)$.   
Using induction on $\varepsilon$, where $1\le \varepsilon\le \xi_{r}$,
it follows that 
\[m_{\t'}\Big(\prod_{k=1}^{\varepsilon-1} T_{\bar\nu_{r}-k,n+j-k}\Big)
= m_{\t'}\Big(\prod_{k=1}^{\varepsilon-1}q T_{\bar\nu_{r}-k,n+j-k\setminus x-k}\Big). \]
Applying a second inductive argument, we find  
\[m_{\t'}\Big(\prod_{k=1}^{\xi_r-1} T_{\bar\nu_{r}-k,n+j-k\setminus x-k}\Big)\Big( \prod_{i=1}^{r-1} \prod_{k=0}^{\xi_{r-i}-1}T_{\bar\nu_{r-i}-k,n+\bar\xi_{r-i}-k}\Big)= m_{\t(x-j,n+j)}.\]
The result follows. 
\end{proof}

Suppose $1\le u\le v\le n$ and that $\pi \in\sym$. 
Let $\mathcal{D}(u,v,\pi)$ be the set of tuples ${\bf p}=(p_0,p_1,\ldots,p_\epsilon)$ such that $u-1 =p_0< p_1<p_2<\ldots<p_\epsilon =v$ and $(p_1)\pi > (p_2)\pi > \ldots > (p_\epsilon)\pi$.    
For each ${\bf p} \in \mathcal{D}(u,v,\pi)$ let $\check{\bf p}$ be the permutation $(p_1,p_1-1,\ldots,p_0+1)(p_2,p_2-1,\ldots,p_1+1)\ldots(p_\epsilon,p_\epsilon-1,\ldots,p_{\epsilon-1}+1)$. Let $\ell({\bf p})=\epsilon-1$ and 
\[b({\bf p}) = \sum_{i=0}^{\epsilon-1} \#\set{j| p_i < j < p_{i+1} \text{ and } (j)\pi > (p_{i+1})\pi}.\]

\begin{lemma} \label{CyclePerm}
Suppose $1\le u\le v\le n$ and that $\pi \in\sym$.  Then
\[ T_{u,v} T_\pi = \sum_{{\bf p} \in \mathcal{D}(u,v,\pi)} q^{b({\bf p})} (q-1)^{\ell({\bf p})} T_{\check{\bf p} \pi}.\]
\end{lemma}

\begin{proof}
We use induction on $v-u$, the case $u=v$ being trivial.  Assume $v-u\ge 1$ and that the lemma holds for $v-u-1$.  By induction,
\[T_{u,v} T_\pi = \sum_{{\bf p} \in \mathcal{D}(u+1,v,\pi)} q^{b({\bf p})} (q-1)^{\ell({\bf p})} T_u T_{\check{\bf p} \pi}.\]
If ${\bf p} =(p_0,p_1,\ldots,p_\epsilon) \in \mathcal{D}(u+1,v,\pi)$ then
\[T_u T_{\check{\bf p} \pi} = \begin{cases}
T_{\check{\bf p'}\pi}, & (u)\pi < (p_1)\pi, \\
q T_{\check{\bf p'}\pi} + (q-1)T_{\check{\bf p''}\pi}, & (u)\pi > (p_1)\pi, 
\end{cases}\]
where ${\bf p'}=(u-1,p_1,\ldots,p_\epsilon)$ and ${\bf p''}=(u-1,p_0,p_1,\ldots,p_\epsilon)$.
The result follows.  
\end{proof}

\subsection{Bumping tableaux} In this section we prove a series of
`bumping lemmas' which culminate in the proof of
Proposition~\ref{Image}. This result contains
Proposition~\ref{ImageDistinct} as a special case, so it completes the
proof of Theorem~\ref{CPDistinctZero}.  Throughout this section,
$\nu$ is an arbitrary partition of $n+\gamma$.  

Suppose that $\t \in \rowstd(\nu)$.  Suppose that $1\le j\le n+\gamma$
and that $\row_\t(j)=r$.  Say that $\s$ is obtained from $\t$ by
\textbf{bumping $j$} down $\t$ if there exists $\epsilon\ge 1$ and
integers $r=r_0<r_1<\ldots<r_{\epsilon}\le z$ and
$j>d_1>\ldots>d_{\epsilon}\ge 1$ such that $\row_{\t}(d_i)=r_i$ for
$1\le i\le \epsilon$ and $\s=\t(j,d_1,\ldots,d_\epsilon)$.  If $\s$ is
such a tableau, write $\s \prec_{j} \t$.  Define
$\ell_{\t}(\s)=\epsilon-1$ and 
\begin{align*}
\bump_\t(\s)&= c^\eta_{r_\epsilon}-\epsilon +\sum_{i=0}^{\epsilon-1} 
         \#\set{j| r_i\le \row_\t(j) < r_{i+1} \text{ and } j>d_{i+1}}\\
      &= c^\eta_{r_\epsilon}-\epsilon +\sum_{i=0}^{\epsilon-1}
                  \s_{[r_i,r_{i+1})}^{>d_{i+1}}.
\end{align*}
The notation $\s_{[r_i,r_{i+1})}^{>d_{i+1}}$ was introduced in Section~\ref{combinatorics}.

\begin{lemma} \label{onnj}
    Suppose $\t\in\rowstd(\nu)$ is such that 
    $\eta=\Shape(\t_{\downarrow n})\ne\mu$ and the 
entries $n+1,n+2,\ldots,n+\gamma$ are in row order.  
Choose $j$ maximal such that $r=\row_{\t}(n+j)<z$.  Then 
\[m_\t (L_{n+j}-[c_r]) = \sum_{\s \prec_{n+j} \t} q^{\bump_\t(\s)} (q-1)^{\ell_\t(\s)} m_\s.\]
\end{lemma}

\begin{proof}
Following Lemma \ref{ArbPerm}, let $\pi$ be the permutation such that $m_\t
= m_{\tnu_\eta} T_\pi$.  Since $\pi \in \sym$ we have that 
$m_\t L_{n+j} = m_{\tnu_\eta}L_{n+j} T_\pi$.  We apply Lemma \ref{onij}, 
keeping the notation of that lemma, except that we set $V=\bar\nu_r-j+1$.
For $V\le g\le n$, let $\sigma_g=\Shape(\t(g,n+j)_{\downarrow n})$.  Then 
\begin{align*} 
m_{\t}(L_{n+j}-[c_r]) & = m_{\tnu_\eta} (L_{n+j}-[c_r]) T_\pi \\
&=q^{\xi_r-1} \sum_{g=V}^n q^{c(g)}m_{\t(g,n+j)} T_\pi \\
& = q^{\xi_r-1} \sum_{g=V}^n q^{c(g)} m_{\tnu_{\sigma_g}} T_{V,g} T_\pi \\
&= q^{\xi_r-1} \sum_{g=V}^n q^{c(g)} \sum_{{\bf p} \in \mathcal{D}(V,g,\pi)} q^{b({\bf p})} (q-1)^{\ell({\bf p})} m_{\tnu_{\sigma_g}} T_{\check{\bf p} \pi} 
\end{align*}
by Lemma \ref{CyclePerm}. 
Now notice that there is a bijection 
\[ \set{\s | \s \prec_{n+j} \t} \overset\sim\longleftrightarrow  
   \set{(g,{\bf p})|V\le g\le n\text{ and }{\bf p}\in\mathcal{D}(V,g,\pi)}
\]
given as follows.  For each pair $(g,{\bf p})$ as above, let ${\bf d} =
(d_1,\ldots,d_\epsilon)$ where $d_i = (p_i)\pi$ for $1\le i < \epsilon$
and $d_\epsilon = (g) \pi$.  By construction, $n+j>d_1>\ldots>d_\epsilon$
and if $1\le i<j\le \epsilon$ then $(p_i)\pi > (p_j)\pi$ and so
$\row_{\t}(i)>\row_{\t}(j)$.  Thus $\s = \t(n+j,d_1,\ldots,d_\epsilon)$ is
formed by bumping $n+j$ down $\t$.  Under this correspondence, since
$\check{\bf p} \pi \in \sym$, in order to see that \[m_{\tnu_{\sigma_g}}
T_{\check{\bf p} \pi} = m_{\t(n+j,d_1,\ldots,d_\epsilon)}\] it is enough to
observe that the permutations $d(\tnu_{\sigma_g}) \check{\bf p} \pi$ and
$(n+j,d_1,\ldots,d_\epsilon)$ agree.  It remains to check that 
\[q^{\xi_r-1}q^{c(g)}q^{b({\bf p})}(q-1)^{\ell({\bf p})} 
           = q^{\bump_\t(\s)}(q-1)^{\ell_\t(s)},\]
which again follows from the definitions.      
\end{proof}

Now suppose that $\T$ is a $\nu$-tableau of arbitrary type which
contains an entry equal to~$k$ in row~$r$.  We generalize the notion
of bumping by saying that a tableau $\U$ is obtained from~$\T$ by
\textbf{bumping $k$ from row $r$} if there exist an integer
$\epsilon\ge 1$ and integers $r=r_0<r_1<\ldots<r_\epsilon\le z$ and
$k>d_1>\ldots>d_\epsilon$ such that for $1\le i\le \epsilon$, row
$r_i$ of $\T$ contains an entry equal to~$d_i$ and~$\U$ is obtained by
repeatedly exchanging $k$ in row $r_i$ with $d_{i+1}$ in row
$r_{i+1}$.  If $\U$ is obtained from $\T$ in this way, write $\U
\prec_{k,r} \T$. We suppress $r$ if $\T$ contains only one entry equal
to $k$.  Define $\ell_{\T}(\U)=\epsilon-1$,
$f_\T^\U=\prod_{i=0}^{\ell_{\T}(\U)}[\U^{d_{i+1}}_{r_i}]$ and
\[\bump_\T(\U) = c^\eta_{r_\epsilon}+\sum_{i=0}^{\epsilon-1}
\Big(\U^{>d_{i+1}}_{r_i} + \U^{\ge d_{i+1}}_{(r_i,r_{i+1})}\Big).\]
This agrees with the previous definition of $\bump_\T(\U)$
when $\T$ is a tableau of type $(1^{n+\gamma})$.  

Define a $\nu$-tableau $\T$ to be \textbf{basic} if it is a
semistandard tableau of type $\eta+1^\gamma$ for some partition $\eta$
of $n$ such that $\eta \subseteq \nu$ and the entries
$z+1,z+2,\ldots,z+\gamma$ are in row order.  Note that for $1\le j\le
\gamma$, the position of~$z+j$ in $\T$ is the same as the position of
$n+j$ in~$\dot\T$.    

\begin{corollary} \label{ONOJ}
  Suppose that $\T$ is a basic tableau of type $\eta+1^\gamma$ such that
  $\eta \neq \mu$.  Let $j$ be maximal such that $r=\row_\T(z+j)<z$.  Then
  $$m_\T\(L_{n+j}-[c_r]\) = 
     \sum_{\U\prec_{n+j}\T} q^{\bump_\T(\U)}(q-1)^{\ell_{\T}(\U)}  
     f_\T^\U m_\U.$$
\end{corollary}

\begin{proof}
Let $\t=\dot\T=\tnu_\eta$, so that $m_\T=m_\t D_\T$ by (\ref{m_S expansion}).
Keeping the notation of Lemma~\ref{onnj} we have
\begin{align*}
m_\T (L_{n+j}-[c_r]) 
  &= m_{\t} (L_{n+j}-[c_r]) D_\T \\
  &= \sum_{\s \prec_{n+j} \t} q^{\bump_\t(\s)} (q-1)^{\ell_\t(\s)} 
          m_\s D_\T.
\end{align*}
Now apply Lemma~\ref{CosetMix} and the definitions.
\end{proof}

\begin{lemma} \label{ONAB}
Suppose that $\T$ is a basic tableau of type $\eta+1^\gamma$ such that $z+j$ lies in 
row~$z$ and that $c\in\Z$.  Then 
$m_\T (L_{n+j}-[c]) = q^c [c_z-c-\gamma+j] m_\T$.
\end{lemma}

\begin{proof}
    It follows from Lemma~\ref{initialmurphy} and the proof of 
    Lemma~\ref{onnj} that 
    $m_\T (L_{n+j}-[c]) = ([c_z-\gamma+j]-[c])m_\T = q^c[c_z-c-\gamma+j] m_\T$. 
\end{proof}

Before generalizing the previous results to bumping tableaux we take a
break and prove the following useful Gaussian integer identity.

\begin{lemma} \label{quantumsum}
Suppose that $v\ge r\ge 0$ and that $C_x,U_x\in\Z$, for $1\le x\le v$.  Then
\[ \sum_{x=r+1}^v \Bigg(\prod_{y=1}^{x-1}q^{U_y}[C_y]\Bigg)[U_x]
       \Bigg(\prod_{y=x+1}^{v}[C_y+U_y]\Bigg)+ \prod_{y=r+1}^v q^{U_y}[C_y] 
       = \prod_{y=r+1}^v[C_y+U_y].\] 
\end{lemma}

\begin{proof}The integer $r$ plays no essential role so we can, and do, assume that
  $r=0$. We claim that for $1\le m\le v$ we have 
\begin{align*}
\sum_{x=m}^{v}\left(\prod_{y=1}^{x-1}q^{U_y}[C_y]\right)&[U_x]\left(\prod_{y=x+1}^{v}[C_y+U_y]\right) + \prod_{y=1}^{v}q^{U_y}[C_y]\\ 
&= \prod_{y=1}^{m-1}q^{U_y}[C_y]\cdot\prod_{y=m}^{v}[C_y+U_y].
\end{align*}
The lemma follows directly from the claim. To prove the claim, we use
downwards induction on $m$.  If $m=v$ then the equation gives
\[\left(\prod_{y=1}^{v-1}q^{U_y}[C_y]\right)[U_{v}]+ \prod_{y=1}^{v}q^{U_y}[C_y] = \left(\prod_{y=1}^{v-1}q^{U_y}[C_y]\right)[C_y+U_y].\]
Now suppose $1\le m<v$ and the claim holds for $m+1$.  Then
\begin{align*}
\sum_{x=m}^{v}&\left(\prod_{y=1}^{x-1}q^{U_y}[C_y]\right)[U_x]\left(\prod_{y=x+1}^{v}[C_y+U_y]\right) + \prod_{y=1}^{v}q^{U_y}[C_y] \\
&= \Bigg(\!\prod_{y=1}^{m-1}q^{U_y}[C_y]\Bigg)[U_m]
\Bigg(\!\prod_{y=m+1}^{v}[C_y+U_y]\Bigg)
{+}\prod_{y=1}^{m}q^{U_y}[C_y]\cdot\!\!\prod_{y=m+1}^v [C_y+U_y]  \\
&= \prod_{y=1}^{m-1}q^{U_y}[C_y]\cdot\prod_{y=m}^{v}[C_y+U_y].
\end{align*}
This completes the proof of the claim and hence the lemma.
\end{proof}

Suppose that $\T$ is a $\nu$-tableau of arbitrary type which contains
an entry equal to $k$ in row~$r$. We say that a tableau $\U$ is
obtained by \textbf{weakly bumping $k$ from row $r$ into row $z$} if
there exist an integer $\epsilon\ge 1$ and integers
$r=r_0<r_1<\ldots<r_\epsilon=z$ and $d_1,d_2,\ldots,d_\epsilon$ such
that for $1\le i\le \epsilon$, we have $k>d_i$ and row $r_i$ of $\T$
contains an entry equal to~$d_i$, and~$\U$ is obtained by repeatedly
exchanging $k$ in row $r_i$ with $d_{i+1}$ in row $r_{i+1}$.  We write
$\U \prec_{k,r}^{\text{w}} \T$. Once again, we suppress $r$ if $\T$
contains only one entry equal to $k$.

\begin{remark}
The differences between bumping $k$ from row $r$ and weakly bumping $k$
from row~$r$ into row~$z$ are that, when $\U \prec_{k,r}^{\text{w}} \T$, we
do not insist that $d_1>d_2>\ldots>d_\epsilon$ but we do insist
that~$r_\epsilon =z$.  
\end{remark}

If $\U\prec_{k,r}^{\text{w}}\T$ then the integers $d_i,r_i$ above are not
necessarily unique.  Nonetheless, there is a unique sequence 
$\a^\U_\T=(a_{r+1},\ldots,a_{z})$; namely, if $r<i\leq z$, define
\begin{equation}\label{a seq}
  a_i = \begin{cases}
             j,& \text{if } \U_i^j =\T_i^j-1 \text{ for some } j, \\
           a_{i+1}, & \text{otherwise}.
        \end{cases}
\end{equation}
(In other words, $\U$ is obtained form $\T$ by moving an entry labeled $k$ from row $r$ to row $z$, then an entry labeled $a_z$ from row $z$ to row $z-1$ and so on, until an entry labeled $a_{r+1}$ is moved from row $r+1$ into row $r$.)
For $r\le i\le z-1$, define 
\[g_\T^\U(i)=\begin{cases}
    [c_z-c_{i}-\gamma+j+\U^{a_{i+1}}_{i}], &\text{if } a_i =a_{i+1}, \\
    [\U^{a^\U_{i+1}}_{i}], &\text{if } a_{i} < a_{i+1} \text{ or } i=r, \\
    q^{c_z-c_{i}-\gamma+j}[\U^{a_{i+1}}_{i}], &\text{if } a_i>a_{i+1}.
\end{cases} \]
Set $g_\T^\U=g_\T^\U(r)\dots g_\T^\U(z-1)$ and if $r \leq x \leq y \leq z$, let
$b_\U^\T(x,y)=\sum_{i=x}^{y-1}\U^{>a_{i+1}}_i$.

\begin{lemma}\label{ONMJ}
Suppose $\T$ is a basic tableau of type $\eta+1^\gamma$ such that $\eta \neq \mu$. Let $j$ be maximal such that $r=\row_\T(z+j)<z$. Then
$$ m_\T\prod_{i=r}^{z-1}\(L_{n+j}-[c_i]\) = q^{c_{r+1}+\dots+c_z-\gamma+j}
    \sum_{\U \prec^{\text{w}}_{z+j}\T} q^{b_\U^\T(r,z)}g_\T^\U m_\U.
$$
\end{lemma}

\begin{proof}
    We use induction on $z-r$ combined with Corollary \ref{ONOJ}.  If
    $r=z-1$ then the result follows from Corollary \ref{ONOJ}.  Now
    suppose that $r<z-1$ and that Lemma \ref{ONMJ} holds for $r<r'\le z$. 
    Let $\calL_{n+j}=\prod_{i=r}^{z-1}\(L_{n+j}-[c_i]\)$. Then by
    Corollary \ref{ONOJ} and induction, it is clear that 
    $m_\T\calL_{n+j}$ is a linear combination of
    terms $m_\U$ where $U \prec_{z+j}^{\text{w}} \T$.

For the remainder of this proof fix a tableau $\U$ such that
$U\prec_{z+j}^{\text{w}}\T$ and let $\a=\a^\U_\T$ be the
sequence defined in (\ref{a seq}) above. Set $a_{z+1}=\infty$ and let 
$v\ge r+1$ be minimal such that $a_{v}<a_{v+1}$.  Define integers 
$r=r_0<r_1<r_2<\ldots<r_s=v$ to be the points at which
$a_{r_\sigma}>a_{r_\sigma+1}$, for $1\le \sigma<s$. Then 
\[a_{r_0+1}=\ldots=a_{r_1}>a_{r_1+1}=\ldots=a_{r_2}>
        \ldots>a_{r_{s-1}+1}=\ldots=a_{r_s},\]
and $a_{r_s}<a_{r_s+1}$. Finally, let 
$R=R_\T^\U=\set{r_\sigma | 1\le\sigma\le s}$.

Suppose that $r+1\le x\le v$.  Then $r_{\epsilon-1}< x\le r_{\epsilon}$
for some $\epsilon=\epsilon(x)$, where $1\le\epsilon\le s$.  Define
integers $r'_0,r'_1,\ldots,r'_\epsilon$ and $d_1,\ldots,d_\epsilon$ by
setting $d_\sigma = a_{r+\sigma}$, for $1\le \sigma\le \epsilon$, and $r'_\sigma=r_\sigma$, for
$0\le \sigma <\epsilon$, and put $r'_{\epsilon}=x$. Now define $\V(x)$ to be the
tableau obtained from~$\U$ by repeatedly exchanging $n+j$ in row $r'_\sigma$
with $d_{\sigma+1}$ in row $r'_{\sigma+1}$.  Then the set of tableaux 
$\set{\V |\U \prec^{\text{w}}_{n+l} \V \prec_{n+l} \T}$ is precisely the
set $\set{\V(x)|r+1\le x\le v}$.  

For this paragraph fix $x$ with $r+1\le x\le v$. For convenience
we set $C_x = c_z-c_{x}-\gamma+j$ and $\U_x = \U^{a_{x+1}}_{x}$. 
Recall that $c^\eta_x=\eta_x-x$, that is, $c^\eta_x=c_x$ for $r+1 \leq x <z$ and $c_z^\eta=c_z-\gamma+j$.  
Then, by
Corollary~\ref{ONOJ}, the coefficient of $m_{\V(x)}$ in 
$m_\T (L_{n+j}-[c_r])$ is 
$$ q^{\bump_\T({\V(x)})}(q-1)^{\epsilon-1}f_\T^{\V(x)}
= q^{c^\eta_{x}+b_r^{x}(\U,\T)}(q-1)^{\epsilon-1} 
        \prod_{\substack{y=r+1\\y \notin R}}^{x-1} q^{\U_y}\cdot
	\prod_{\sigma=1}^{\epsilon-1}[\U_{r_\sigma}].
$$
If $x\ne z$ then, by induction, the coefficient of $m_\U$ in 
$m_{\V(x)}\prod_{i=x}^{z-1}(L_{n+j}-[c_i])$ is 
\[q^{c_{x+1}+\dots+c_z-\gamma+j+b_\U^\T(x,z)}[\U_x]
                \prod_{\tau=x+1}^{z-1}g_\T^\U(\tau).\]
Finally, by Lemma~\ref{ONAB}, 
\[m_{\U} \prod_{i=r+1}^{x-1} (L_{n+j}-[c_i]) 
= q^{c_{r+1}+\dots+c_{x-1}}\prod_{y=r+1}^{x-1} [C_y] m_{\U}.\]

As already noted,
$\set{\V |\U \prec^{\text{w}}_{n+l} \V \prec_{n+l} \T}
    =\set{\V(x) | 1\le x\le v}$.  Assume now that $v \neq z$; the case $v=z$ is similar but contains some technical differences which we leave to the reader.  
Collecting the terms above, the coefficient of
$q^{c_{r+1}+\dots+c_z-\gamma+j+b_\U^\T(r,z)}m_\U$ in~$m_\T\calL_{n+j}$ is
\begin{align*}
\sum_{x=r+1}^v&(q-1)^{\epsilon(x)-1}[\U_x]
     \prod_{\substack{y=r+1\\y\notin R}}^{x-1}q^{\U_y}\cdot
     \prod_{\sigma=1}^{\epsilon(x)-1}[\U_{r_\sigma}]\cdot
     \prod_{\tau=x+1}^{z-1}g_\T^\U(\tau)\cdot
     \prod_{y=r+1}^{x-1}[C_y]\\
  &=\prod_{y=v+1}^{z-1}g_\T^\U(y) \cdot \Bigg\lbrace\sum_{x=r+1}^v \ \ [\U_x]
     \prod_{\substack{y=r+1\\y\notin R}}^{x-1}q^{\U_y}[C_y]\cdot
     \prod_{\sigma=1}^{\epsilon(x)-1}(q^{C_{r_\sigma}}-1)[\U_{r_\sigma}]\cdot
     \!\!\prod_{\tau=x+1}^{v}g_\T^\U(\tau)\Bigg\rbrace,
\end{align*}
where the last equation follows by rearranging the terms using the
identity $(q-1)[C]=q^C-1$, for any $C\in\Z$.
For $1\le x\le v$ set
$$h(x)=[\U_x]\prod_{\substack{y=r+1\\y\notin R}}^{x-1}q^{\U_y}[C_y]\cdot
     \prod_{\sigma=1}^{\epsilon(x)-1}(q^{C_{r_\sigma}}-1)[\U_{r_\sigma}]\cdot
     \prod_{y=x+1}^{v}g_\T^\U(y).$$

To complete the proof of the lemma we need to show that 
$\sum_{x=r+1}^v h(x)=\prod_{x=r}^{v} g_\T^\U(x)$. Hence, it is enough to
establish the following claim and then set $\epsilon=1$:

\begin{claim}
Suppose that $1\le\epsilon\le s$. Then 
\[ \sum_{x=r_{\epsilon-1}+1}^{v} h(x) = 
      \prod_{\substack{y=r+1\\y \notin R}}^{r_{\epsilon-1}}q^{\U_y}[C_y]\cdot
      \prod_{\sigma=1}^{\epsilon -1}(q^{C_{r_\sigma}}-1)[\U_{r_\sigma}]\cdot
      \prod_{\tau=r_{\epsilon-1}+1}^{v}g_\T^\U(\tau)\cdot
\]
\end{claim}

We prove the claim by downwards induction on $\epsilon$.  If $\epsilon=s$  
then $\epsilon(x)=s$, for $x=r_{s-1}+1,\dots, r_s=v$, so
$$\sum_{x=r_{s-1}+1}^{v} h(x) 
=\sum_{x=r_{s-1}+1}^v \ \ [\U_x]
     \prod_{\substack{y=r+1\\y\notin R}}^{x-1}q^{\U_y}[C_y]\cdot
     \prod_{\sigma=1}^{s}(q^{C_{r_\sigma}}-1)[\U_{r_\sigma}]\cdot
     \!\!\prod_{\tau=x+1}^{v}g_\T^\U(\tau).$$
Consulting the definitions reveals that for $r+1 \leq y \leq v$ we have
$$g_\T^\U(y) = \begin{cases}
        [U_y],&\text{if }y=v, \\
        q^{C_y}[\U_y],&\text{if }v \neq y\in R,\\
        [C_y+\U_y],&\text{if }y\notin R.
      \end{cases}$$
Therefore, 
\begin{align*}
\sum_{x=r_{s-1}+1}^{v} h(x) 
  &=[U_v] \cdot \prod_{\substack{y=r+1\\y \notin R}}^{r_{\epsilon-1}}q^{\U_y}[C_y]\cdot
    \prod_{\sigma=1}^{s-1}(q^{C_{r_\sigma}}-1)\Bigg\lbrace
    \prod_{y=r_{s-1}+1}^{v-1}q^{\U_y}[C_y]\\
  &\hspace*{15mm}+\sum_{x=r_{s-1}+1}^{v-1}
     \prod_{y=r_{s-1}+1}^{x-1}q^{\U_y} [C_y]\cdot[\U_x]\cdot
     \prod_{y=x+1}^{v-1}[C_y+\U_y]\Bigg\rbrace \\
  &=[U_v]\cdot\prod_{\substack{y=r+1\\y \notin R}}^{r_{\epsilon-1}}q^{\U_y}[C_y]\cdot
    \prod_{\sigma=1}^{s}(q^{C_{r_\sigma}}-1)\cdot
    \prod_{i=r_{s-1}+1}^{v-1}[C_i+\U_i]
\end{align*}
by Lemma \ref{quantumsum}. This proves the claim when $\epsilon=s$. The
proof of the claim when $\epsilon<s$ follows easily by induction
using a similar argument, so we leave the details to the reader.
\end{proof}

\begin{corollary} \label{CONMJ}
Suppose that $\T$ is a basic tableau and that $j\in[1,\gamma]$ is an integer 
such that either $j=\gamma$ or $\row_\T(z+j+1)=z$. 
Let $r=\row_\T(n+j)$ and fix $y$ with $1\le y\le r$. If $r=z$ then
\[m_\T \prod_{i=y}^{z-1}(L_{n+j}-[c_i]) 
        =q^{c_1+\dots+c_{z-1}} \prod_{i=y}^{z-1}[c_z-c_i-\gamma+j]\, m_\T.\] 
If $r<z$ then
\[m_\T \prod_{i=y}^{z-1}(L_{n+j}-[c_i]) 
= q^{c_{1}+\dots + c_z-c_r+j-\gamma}\prod_{i=y}^{r-1}[c_z-c_i-\gamma+j]
  \sum_{\U\prec^{\text{w}}_{n+j}\T} q^{b_\U^\T(r,z)}g_\T^\U\, m_\U.\]
\end{corollary}

\begin{proof}
This is an immediate consequence of Proposition \ref{ONMJ} and 
Lemma~\ref{ONAB}.
\end{proof}

The next result will complete the proof of Theorem~\ref{CPDistinctZero}.
Although we could prove this result for a slightly more general class of
partitions, we assume that $\nu_i-\nu_{i+1}\ge\gamma$, for $1\le i<z$,
because this assumption significantly simplifies the notation that we need.

Suppose $\t=\tnu_\eta$ is an almost initial tableau.     
Choose $k$ with $1\le k\le \gamma$ and let $\eta^{(k)}$ be the partition of $n$ given by
\[\eta^{(k)}_i = \begin{cases}
\eta_i+\t_i^{> n+\gamma-k} , & 1\le i < z, \\
\nu_i-k, & i=z. 
\end{cases}\]
Write $\U \inim{k} \t$ if $\U \in \mathcal{T}_0(\nu,\eta+1^\gamma)$
and $\Shape(\U_{\downarrow z})=\eta^{(k)}$ and the numbers
$z+1,z+2,\ldots,z+\gamma$ in $\U$ are in row order. 

\begin{proposition}\label{Image}
Assume that $\nu_i-\nu_{i+1} \geq \gamma$, for $1 \leq i<z$, and that $\t=\tnu_\eta$ is an almost initial tableau.  Suppose that $1 \leq k \leq \gamma$ and that
$1\le y\le\row_\t(n+\gamma-k+1)$.  Then   
\begin{equation*}
m_{\t} \prod_{i=y}^z \prod_{j=1}^k (L_{n+\gamma-j+1}-[c_i]) =q^{c(k)} \sum_{\U \inim{k} \t}  
    \Bigg(\prod_{i=y}^{z-1}[\U^{(i,z]}_i]^! 
          \prod_{j=0}^{k-\U^{(i,z]}_i-1}[c_z-c_i-j]\Bigg)  m_\U
\end{equation*}
where $$c(k)=\sum_{i=y}^z k c_i +\t^{>n+\gamma-k}_i \Big(\t_i^{(n,n+\gamma-k]} - \t^{>n+\gamma-k}_{>i}-c_i\Big).$$
\end{proposition}

\begin{proof} For the duration of the proof we set 
$\calL'_{k'}=\prod_{i=y}^z\prod_{j=1}^{k'}(L_{n+\gamma-j+1}-[c_i])$, for $1 \leq k' \leq k$.  
Then we have to compute $m_\t\calL'_k$.
First note that if $\T$ is the basic tableau obtained by replacing each
entry $x$ with $1\le x\le n$ in $\t$ by its row index in $\t$ and each
entry $n+1\le x\le n+\gamma$ with $x-n+z$ then $m_\t=m_\T$ by (\ref{m_S
expansion}). If $k=1$ or $\row_\t(n+\gamma-k+1)=z$ then the
result follows from Corollary \ref{CONMJ}.  So suppose that
$1<k\le\gamma$ and that $\row_\t(n+\gamma-k+1)=r<z$. By induction on~$k$ we can assume that the Proposition holds for $m_\t\calL'_{k'}$ whenever $1\le k'<k$.     

Repeated applications of Corollary \ref{CONMJ} shows that
$m_\t\calL'_k=m_\T\calL'_k$ is a linear combination of terms $m_\U$, where 
$\U\inim{k}\t$. That each tableau $\U$ is semistandard follows because
$\nu_i-\nu_{i+1}\ge \gamma$ for all $i$. We now fix $\U$ with 
$\U\inim{k}\t$ and compute the coefficient of $m_\U$ in $m_\T\calL'_k$.

Suppose that~$V$ is a basic tableau such that 
$\U\prec^{\text{w}}_{n+\gamma-k+1}\V\inim{k-1}\t$. By Corollary
\ref{CONMJ}, the coefficient of $m_\U$ in 
$m_\V \prod_{i=y}^{z-1}(L_{n+\gamma-k+1}-[c_i])$ is
\[q^{c_{1}+\dots+c_{r-1}+c_{r+1}+\ldots+c_z+b_r^z(\U,\V)-k+1}g_\V^\U
     \prod_{i=y}^{r-1}[c_z-c_i-k+1].\]
By induction, the coefficient of $m_\V$ in $m_{\t}\calL'_{k-1}$ is  
\[q^{c(k-1)}\prod_{i=y}^{z-1}\Big([\V^{(i,z]}_i]^!
     \prod_{j=0}^{k-\V^{(i,z]}_i-1}[c_z-c_i-j]\Big).\]
Now observe that 
\[c(k)=c(k-1) +c_1+\ldots+c_z-c_r+\t^{(n,n+\gamma-k]}-k+1.\]
Therefore, the coefficient of 
$q^{c(k)}m_\U$  in $m_\t\calL'_k$ is
$$\sum_{\substack{\V\in\SStd(\nu,\eta+1^\gamma)\\
                \U\prec^{\text{w}}_{n+\gamma-k+1}\V\inim{k-1}\t}}\!\!\!\!\!\!
     q^{\t_r^{(n,n+\gamma-k]}+b_r^z(\U,\V)} g_\V^\U \prod_{i=y}^{r-1}[c_z-c_i-k+1]\cdot
     \prod_{i=y}^{z-1}\Big([\V^{(i,z]}_i]!\!\!
         \prod_{j=0}^{k-\V^{(i,z]}_i-1}\!\![c_z-c_i-j]\Big).
$$
Consulting the definitions, if $\V\in\SStd(\nu,\eta+1^\gamma)$ and
$\U\prec^{\text{w}}_{n+\gamma-k+1}\V\inim{k-1}\t$ then
\[\V^{(i,z]}_i = \begin{cases}
\U^{(i,z]}_i, &  1\le i\le r-1, \text{ or } r+1\le i\le z \text{ and }r+b_i\ne i, \\
\U^{(i,z]}_i-1, & i=r, \text{ or } r+1\le i\le z \text{ and }b_i = i,
\end{cases}\]
whenever $1\le i\le z$. This allows us to rewrite the last equation in
terms of $\U$. Before we do this, however, we change the indexing set for
the sum to something that is more manageable.

Suppose that $\U\prec^{\text{w}}_{n+\gamma-k+1}\V$. Then $\V$ is completely
determined by a sequence $\a^\U_\V=(a_{r+1},\dots,a_z)$ as in (\ref{a
seq}). Let $\A=\set{\a=(a_{r+1},\ldots,a_z)|i\le a_i\le z\text{ for }r\le
i\le z}$. Then $\a^\U_\V\in\A$ for each tableau $\V$ in the sum above.
Conversely, if $\a\in\A$ and $\a$ does not correspond to one of the tableau
above then there exists an $i$, with $r\le i\le z-1$, such that $a_i\ne
a_{i+1}$ and $\U^{a_{i+1}}_i=0$.  Therefore, $h_\U^\a(i)=0$, where we define
\[h_\U^{\a }(i)=
\begin{cases}
[C_i+\U^{a_{i+1}}_i][\U^{(i,z]}_i], 
              &\text{if } i\ne a_i = a_{i+1}, \\
[C_i+\U^{(i,z]}_i][\U^{a_{i+1}}_i], 
              &\text{if } i=a_i < a_{i+1},\text{ or if } i=r, \\
[\U^{a_{i+1}}_i][\U^{(i,z]}_i], 
              &\text{if } i\ne a_i < a_{i+1}, \\
q^{C_i}[\U^{a_{i+1}}_i][\U^{(i,z]}_i], 
              &\text{if } i\ne a_i > a_{i+1}.
\end{cases}\]
where $C_i = c_z-c_i-k+1$, for $r\le i<z$.  Recall that
$b_r^z(\U,\V)=\sum_{i=r}^{z-1}\U_i^{>a_{i+1}}= \sum_{i=r}^{z-1} \U_i^{(a_{i+1},z]} +\t^{(n,n+\gamma-k]}_r$.  Therefore, by comparing the
definitions of $g_\U^\V(i)$ and $h_\U^\a(i)$, and observing that
$\V_i^{(i,l)}\le\U_i^{(i,l)}-1$, the coefficient of 
$q^{c(k)}m_\U$ in $m_\t\calL'_k$ given above becomes
$$
     \prod_{i=y}^{r-1}[C_i]\cdot
     \prod_{i=y}^{z-1}\Big([\U^{(i,z]}_i-1]^!
         \prod_{j=0}^{k-\U^{(i,z]}_i-2}[c_z-c_i-j]\Big).
         \sum_{\a\in\A}\prod_{i=r}^{z-1}q^{\U_i^{(a_{i+1},z]}}h_\U^\a(i)
$$
where we adopt the convention that $[-1]^!=1$. 
By definition, $\U_i^{(i,z]}=0$, for $1\le i<r$, and
$C_i+\U_i^{(i,z]}=c_z-c_i-(k-\U_i^{(i,z]}-1)$, for $1\le i<z$. Therefore,
to complete the proof we need to show that 
$$\sum_{\a\in\A}\prod_{i=r}^{z-1}q^{\U_i^{(a_{i+1},z]}}h_\U^\a(i)
             =\prod_{i=r}^{z-1}[\U_i^{(i,z]}][C_i+\U_i^{(i,z]}].$$
This will follow once we have established the following claim by
setting $x=r$ and, for definiteness, $a=r$.

\begin{claim}
 Let 
$\A_{a,x}=\set{(a,a_{x+1},\ldots,a_z)|i\le a_{i}\le z\text{ for }x+1\le i\le z}$ where $r\le x\le z-1$ and $x\le a\le z$. 
Then
\[\sum_{\a \in \A_{a,x}}\prod_{i=x}^{z-1} q^{\U^{(a_{i+1},z]}_i} h_\U^\a(i) 
                 = \prod_{i=x}^{z-1}[\U^{(i,z]}_i][C_i+\U^{(i,z]}_i].\]
\end{claim}
To prove the claim, we use downwards induction on $x$.  If $x=z-1$ then $b=z-1$ or $b=z$.  If $a=z-1$ or $x=r$ then 
\[\sum_{\a \in \A_{a,x}}\prod_{i=x}^{z-1} q^{\U^{(a_{i+1},z]}_i} h_\U^\a(i) = [C_{z-1}+\U^{[z,z]}_{z-1}][\U^{z}_{z-1}], \]
and if $a=z$ and $x\ne r$ then 
\[\sum_{\a \in \A_{a,x}}\prod_{i=x}^{z-1} q^{\U^{(a_{i+1},z]}_i} h_\U^\a(i) = [C_{z-1}+\U^z_{z-1}][\U^{[z,z]}_{z-1}].\]
Since $\U^z_{z-1}=\U^{[z,z]}_{z-1}$, the claim holds for $x=z-1$.  So suppose $r+1\le x <z-1$ and the claim holds for $x+1$.  
\begin{align*}
\sum_{\a \in \A_{a,x}}\prod_{i=x}^{z-1} q^{\U^{(a_{i+1},z]}_i} h_\U^\a(i) 
& =  \sum_{a_{x}=x+1}^{z} q^{\U^{(a_{x+1},z]}_x}h_\U^\a(x) 
\sum_{\a \in \A_{a,x+1}} \prod_{i=x+1}^{z-1}q^{\U^{(a_{i+1},z]}_i}h_\U^\a(i)\\
& = \prod_{i=x+1}^{z-1}[\U^{(i,z]}_i][C_i+\U^{(i,z]}_i]  
    \sum_{a_{x+1}=x+1}^{z} q^{\U^{(a_{x+1},z]}_x}h_\U^\a(x)
\end{align*}
by induction. If $a = x$ or $x=r$ then
\begin{align*}
    \sum_{a_{x+1}=x+1}^{z} q^{\U^{(a_{x+1},z]}_x}h_\U^\a(x)
    & = \sum_{a_{x+1}=x+1}^z q^{\U^{(a_{x+1},z]}_x} [\U^{x_{a+1}}_x] \\
&= [\U^{(x,z]}_x].
\end{align*}
If $a\ne x$ and $x\ne r$ then
$\sum_{a_{x+1}=x+1}^{z} q^{\U^{(a_{x+1},z]}_x}h_\U^\a(x)$ is equal
to
\begin{align*}
[\U^{(x,z]}_x] \Bigg(\sum_{i=x+1}^{a-1}q^{C_x}q^{\U^{(i,z]}_x}&[\U^i_x]
    + q^{\U^{(a,z]}_x}[C_x+\U^{a}_x] 
    + \sum_{i=a+1}^{z} q^{\U^{(i,z]}_x} [\U^{i}_x]\Bigg)\\
&= [\U^{(x,z]}_x] [C_x + \U^{(x,z]}_x]
\end{align*}
This completes the proof of both the claim and the Proposition.
\end{proof}

As Proposition~\ref{ImageDistinct} is a special case of
Proposition~\ref{Image},  this completes the proof of
Theorem~\ref{CPDistinctZero} and, in fact, all of our main results when $F$ is a
field of characteristic zero.

\subsection{Gaussian integer division}
In this section we prove Lemma~\ref{divides1} which were used in
Section~2 to define the polynomials $\beta_{\la\mu}(q)$ in (\ref{CP
hom}). Therefore, the results in this subsection complete the proof of
our main results when $F$ is a field of positive characteristic.
Accordingly, we assume that $F$ is a field of characteristic $p>0$,
that $e>1$ and that $\zeta$ is a primitive $e^{\text{th}}$ root of
unity in $F$.   

Let $K=F(\q)$, where $\q$ is an indeterminate over $F$.  For $l \in \Z$, set $[l]_{\q}=\frac{\q^l-1}{\q-1} \in K$.  Set $[0]_{\q}^!=1\in K$ and for $l\ge 1$ set $[l]_{\q}^! = [l-1]_{\q}^![l]_{\q}$.   
For $l \in \Z \setminus \{0\}$, define $\nu_p(l)$ to be the largest integer $v\ge 0$ such that $p^v$ divides $l$ (in $\Z$) and set
\[\nu_{e,p}(l)= \begin{cases} 
    0, & \text{if }e \nmid l, \\
1+\nu_{p}(\frac{l}{e}), & \text{otherwise}.
\end{cases}\]

\begin{lemma} \label{FirstDivideLemma}
Suppose that $r\ge1$ and that $(a_1,a_2,\ldots,a_r)$ and 
$(b_1,b_2,\ldots,b_r)$ are two $r$-tuples of non-zero integers such that
$\nu_{e,p}(a_j)\ge \nu_{e,p}(b_j)$, for $1\le j\le r$.  Then there exist 
polynomials $f(\q),g(\q) \in F[\q,\q^{-1}]$ such that $g(\zeta)\ne 0$ and 
\[\frac{\prod_{j=1}^r [a_j]_{\q}}{\prod_{j=1}^r[b_j]_{\q}} = \frac{f(\q)}{g(\q)}.\]
\end{lemma}

\begin{proof}
It is sufficient to show that if $a,b \in \Z\setminus\{0\}$ and
$\nu_{e,p}(a)\ge \nu_{e,p}(b)$ then $[a]_{\q}/[b]_{\q}$ can be written in
this form.  Since $[l]_{\q}=-\q^l[-l]_{\q}$, we may assume that $a,b >0$.
If $\nu_{e,p}(b)=0$ then $[a]_{\q}/[b]_{\q}$ itself is of the correct form.
So take $a=xep^k, b=yep^l$ where $p \nmid x,y$ and $k\ge l$.  Then
\[\frac{[a]_{\q}}{[b]_{\q}}
    =\frac{1+\q+\ldots+\q^{a-1}}{1+\q+\ldots+\q^{b-1}}
    =\frac{1+\q^{ep^l}+\ldots+\q^{(xp^{k-l}-1)ep^l}}
          {1+\q^{ep^l}+\ldots+\q^{(y-1)ep^l}}.
\]
Since $\zeta$ is an $e^{\text{th}}$ root of unity and $p \nmid y$, the value of
the denominator of the right hand term at $\zeta$ is non-zero.  
\end{proof}

\begin{lemma} \label{CMgamma1}
Suppose that $K,\gamma,m > 0$.  For any integer $l$ define $l'$ by writing $l=l^\ast m + l'$ where $0\le l' <m$.  
Let $C= -K$.  For $0\le X\le \gamma$, let $\M_X$ be the multiset $\{1,2,\ldots,X,K,K+1,\ldots,K+\gamma-X-1\}$ and let $N(X)$ be the number of elements of $\M_X$ which are divisible by $m$.  Then
\[N(X) = \begin{cases}
\max\left\{0,\left\lceil\frac{\gamma-C'}{m}\right\rceil\right\}, & X' < (\gamma+K)', \\[5pt]
\max\left\{0,\left\lfloor\frac{\gamma-C'}{m}\right\rfloor\right\}, & X'\ge(\gamma+K)'. \\
\end{cases}     \]
\end{lemma}

\begin{proof}
By definition, $N(X)$ is equal to the number of elements of
$\{K,K+1,\ldots,K+\gamma-X'-1\}$ which are divisible by $m$.  It is then
straightforward to check that this 
$$N(X)=\max\left\{0,\left\lceil \frac{\gamma-C'-X'}{m}\right\rceil\right\}.$$  
Noting that $(\gamma-C')' = (\gamma+K)'$, the result follows.    
\end{proof}

\begin{lemma} \label{CMgamma2}
  Suppose $K>0$ and $\gamma\ge e$.  For  $0\le X\le \gamma$, let $\M_X$ be
  the multiset \[\M_X=\{1,2,\ldots,X,K,K+1,\ldots,K+\gamma-X-1\}.\] For
  $i\ge 0$, set $N(X)_i=\#\set{x \in \M_X|\nu_{e,p}(x)\ge i}$. Let $s$
  be maximal such that $\gamma\ge e p^s$ and $A$ minimal such that 
  $Aep^s\ge K$ and set $\beta=\gamma-Aep^s+K$, so that
  $$\M_\beta=\{1,2,\ldots,\gamma-Aep^s+K,K,K+1,\ldots,Aep^s-1\}.$$
  Then $0\le \beta\le \gamma$ and if $0\le X\le\gamma$ then
  $N(\beta)_i\le N(X)_i$, for all $i\ge 0$.  
\end{lemma}

\begin{proof} That $0\le\beta\le\gamma$ is clear from the definitions. To
  prove the second claim $i\ge 0$. For any integer $l\ge 0$ define $l'$ by
  $l=l^\ast ep^i + l'$ where $0\le l'< ep^i$.  By Lemma \ref{CMgamma1}, to
  show that $N(\beta)_i\le N(X)_i$ whenever $0\le X\le \gamma$ it is
  sufficient to prove that $\beta'\ge (\gamma +K)'$. In fact, our choice of
  $\beta$ gives $\beta'=(\gamma+K)'$.  
\end{proof}

\begin{corollary} \label{CMgamma3}
Suppose that $\gamma >0$ and $C<0$.  For $0\le X\le \gamma$, let $\M_X$
denote the multiset $\M_X=\{1,2,\ldots,X,C,C-1,\ldots,C-\gamma+X+1\}$. For
$i\ge 0$ let $$N(X)_i = \#\set{x \in \M_X | \nu_{e,p}(x)\ge i}.$$ 
Then there exists an integer $\beta$ with $0\le \beta\le \gamma$ such that
$N(\beta)_i\le N(X)_i$ whenever $0\le X\le \gamma$ and $i\ge 0$.   
\end{corollary}

\begin{proof}
If $\gamma<e$ then set $\beta=\gamma$.  Otherwise set $K=-C$.  Then for all
$i$, $N(X)_i$ is the number of elements $x \in
\{1,2,\ldots,X,K,K+1,\ldots,K+\gamma-X-1\}$ such that $\nu_{e,p}(x)\ge i$.
Hence, the result follows from Lemma \ref{CMgamma2}.   
\end{proof}

\begin{lemma} \label{divides1}
Suppose that $\gamma >0$ and that $C<0$.  Write $\gamma = \gamma^\ast e +
\gamma'$ where $0\le \gamma' <e$.  Then there exists an integer $\beta$,
with $0\le \beta\le \gamma$, and polynomials 
$f_X(\q),g_X(\q) \in F[\q,\q^{-1}]$ such that $g_X(\zeta)\ne 0$ and
\[\frac{[X]_{\q}^!\prod_{j=0}^{\gamma-X-1}[C-j]_{\q}}{[\beta]_{\q}^!
                 \prod_{j=0}^{\gamma-\beta-1}[C-j]_{\q}} 
                 = \frac{f_X(\q)}{g_X(\q)},\]
whenever $0\le X\le \gamma$.  Moreover, if 
$C \equiv 0 \mod ep^{\ell_p(\gamma^\ast)}$ then $\beta =\gamma$ and 
$f_X(\zeta)\ne 0$ if and only if $X=\gamma$.  
\end{lemma}

\begin{proof} 
Using the notation of Corollary \ref{CMgamma3}, there exists an integer
$\beta$ with $0\le \beta\le \gamma$ such that $N(\beta)_i\le N(X)_i$ for
all $i\ge 0$.  Therefore it is possible to reorder the elements in the
multisets $\M_X=\{x_1,x_2,\ldots,x_\gamma\}$ and
$\M_{\beta}=\{b_1,b_2,\ldots,b_\gamma\}$ in such a way that 
$\nu_{e,p}(x_j)\ge \nu_{e,p}(b_j)$, for $1\le j\le \gamma$. Hence, by 
Lemma~\ref{FirstDivideLemma}, there exists an integer $\beta$ with the
required properties.    

Now suppose that $C \equiv 0 \mod ep^{\ell_p(\gamma^\ast)}$.  Note that $ep^{\ell(\gamma^\ast)} > \gamma$.  By Lemma \ref{CMgamma2}, we may take $\beta=\gamma$.   Now, suppose $X\ne \beta$.  Reorder $\M_X$ and $\M_\beta$ as above so that $\nu_{e,p}(x_j)\ge \nu_{e,p}(b_j)$ for $1\le j\le \gamma$.  Assume that $x_1=C$.  By Lemma \ref{FirstDivideLemma}
\[\frac{\prod_{j=1}^\gamma [x_j]}{\prod_{j=1}^\gamma [b_j]}
             =\frac{[C]_{\q}f'_X(\q)}{[b_1]_{\q}g'_X(\q)}\]
for some $f'_X(\q),g'_X(\q)\in F[\q,\q^{-1}]$ with $g'_X(\zeta)\ne 0$.  Since $1\le b_1\le \gamma$, we have $\nu_{e,p}(b_1)<\nu_{e,p}(C)$.  Consider $[C]_{\q}/[b_1]_{\q}$.  If $e \nmid b_1$ then the evaluation of $[C]_{\q}$ at $\zeta$ is zero.  Otherwise, write $-C=xep^k,b_1=yep^l$ where $p \nmid x,y$ so that $k>l$.   Then
\[\frac{[C]_{\q}}{[b_1]_{\q}} = \frac{-\q^{-C}(1+\q^{ep^l}+\ldots+\q^{(xp^{k-l}-1)ep^l})}{1+\q^{ep^l}+\ldots+\q^{(y-1)ep^l}}.\]
Since $p \mid xp^{k-l}$, the numerator of the last term evaluated at $\zeta$ is zero.     
\end{proof}

\section*{Acknowledgments}
We thank John Murray for extended discussions about his work with Harald
Ellers~\cite{EM,EllersMurray} on Carter-Payne homomorphisms for symmetric
groups, on which this paper is based. We also thank Steve Donkin for
telling us about the results in Dixon's thesis~\cite{Dixon} and Anton
Cox for his helpful comments.   

Research on this paper was begun at the Mathematical Sciences Research
Institute in Berkeley in 2008 during the parallel programs `Combinatorial
representation theory' and `Representation theory of finite groups and
related topics'. The authors thank the MSRI and the organizers of these
programs for  their support. This work was supported, in part, by the
Australian Research Council.

\end{document}